%% file: main.tex
\documentclass{article}
\input{preamble.tex}

\usepackage{morefloats}

\usepackage{scalerel}
\usepackage[ruled]{algorithm2e}

\usepackage{url}
\usepackage{float}
\usepackage[backend=biber,style=alphabetic,sorting=nyt,maxbibnames=99,backref=true]{biblatex}
\usepackage{etoolbox}
\apptocmd{\sloppy}{\hbadness10000\relax}{}{}

\usepackage{tikz}
\usetikzlibrary{arrows,decorations.markings,decorations.pathmorphing,backgrounds,positioning,fit,petri,knots,shapes,automata, arrows.meta, positioning, decorations.pathreplacing,calligraphy}
\usetikzlibrary{hobby}

\begin{document}

\author{Nicol\'as Bitar}

\date{}

\title{Realizability of Subgroups by Subshifts of Finite Type}
\maketitle
\begin{abstract}
\hskip -.2in
\noindent
We study the problem of realizing families of subgroups as the set of stabilizers of configurations from a subshift of finite type (SFT). This problem generalizes both the existence of strongly and weakly aperiodic SFTs. We show that a finitely generated normal subgroup is realizable if and only if the quotient by the subgroup admits a strongly aperiodic SFT. We also show that if a subgroup is realizable, its subgroup membership problem must be decidable. The article also contains the introduction of periodically rigid groups, which are groups for which every weakly aperiodic subshift of finite type is strongly aperiodic. We conjecture that the only finitely generated periodically rigid groups are virtually $\Z$ groups and torsion-free virtually $\Z^2$ groups. Finally, we show virtually nilpotent and polycyclic groups satisfy the conjecture. 
\vskip .1in
\noindent\keywords{Subshift of finite type, symbolic dynamics, realizability, aperiodic tilings.}

\end{abstract}

%

\section{Introduction}

For a finite alphabet $A$ and a group $G$, a subshift is a closed $G$-equivariant subset of the compact space $A^G$. These sets are used to model dynamical systems~\cite{morse1938symbolic, barbieri2024effective} and as computational models~\cite{aubrun2017effectiveness}. Subshifts also have an equivalent combinatorial definition: a subshift is the set of all configurations from $A^G$ that respect local rules given by forbidden patterns. Given a set of forbidden patterns $\mathcal{F}$, the subshift $\X_{\mathcal{F}}$ it defines is the set of configurations $x\in A^G$ that avoid all patterns from $\mathcal{F}$. A subshift of finite type (SFT) is a subshift $X$ such that there exists a finite set of forbidden patterns $\Fo$ that defines it, i.e. $X = \X_\Fo$. SFTs form an interesting class of subshifts since they are specified by a finite amount of information, and define models of computation whose computational power depends on the group~$G$~\cite{aubrun2018domino}.


There have been recent interest for two particular types of SFTs: weakly and strongly aperiodic SFTs. A subshift is said to be weakly aperiodic if the orbit of all its configurations are infinite, and strongly aperiodic if the action of the group on the subshift is free. The goal has been to understand which properties of the underlying group influence the existence of both types of aperiodic subshifts. For instance, the existence of strongly aperiodic subshifts has been found to be connected to the large-scale geometry of the group~\cite{cohen2017large}, and the computability of its word problem~\cite{jeandel2015aperiodic2}. A state of the art on the existence of these two types of aperiodic SFTs is presented in Sections~\ref{sec:strongly} and~\ref{sec:weakly}\\

To further understand the influence of the group on the aperiodicity and periods of its subshifts of finite type, in this article we study the problem of realizability of families of subgroups as the set of stabilizers of subshifts of finite type. This problem generalizes the existence of both weakly and strongly aperiodic SFTs; the former being the case when the family of subgroups contains only infinite index subgroups, and the latter the case when the family is the singleton containing the trivial subgroup. The two main instances of the problem we are interested in is the realizability of singletons -- when all stabilizers are the same subgroup -- and the realizability of non-trivial families of infinite index subgroups, that is, the existence of weakly aperiodic SFTs that are not strongly aperiodic.

\subsection*{Realizability of Subgroups}

Let $\Sub(G)$ denote the space of subgroups of $G$. We say a family of subgroup $\G\subseteq\Sub(G)$ is \define{realizable} is there exists a non-empty SFT $X\subseteq A^G$ such that the set of stabilizers of $X$ is equal to $\G$. He say a subgroup $H\in\Sub(G)$ is realizable if the family $\G = \{H\}$ is realizable. The main questions of this article are the following.
\begin{question}
	Which subsets of $\Sub(G)$ are realizable? Which subgroup of $G$ are realizable?
\end{question}

For subgroups, it is possible to quickly rule out the realizability of subgroup which are not normal (Lemma~\ref{lem:no_normal}). Furthermore, we obtain a characterization of finitely generated normal subgroups that are realizable.
\begin{thmx}
\label{intro:ida_y_vuelta}
    Let $N\trianglelefteq G$ be a non-trivial finitely generated normal subgroup. Then, $N$ is realizable in $G$ if and only if $G/N$ admits a strongly aperiodic SFT.
\end{thmx}

This theorem is a generalization of the well-known phenomenon that occurs in $\Z^2$, where any subgroup isomorphic to $\Z$ is not realizable (see Lemma~\ref{lem:OGPR}). The proof of Theorem~\ref{intro:ida_y_vuelta} also allows us to show that if we lift the SFT restriction, any normal subgroup of $G$ is realizable by a subshift on $G$ (Proposition~\ref{prop:no_reservations}).\\

For recursively presented finitely generated groups, we also obtain a connection between a subgroups realizability and the decidability of its membership problem.
\begin{thmx}
\label{intro:computational}
    Let $G$ be a finitely generated recursively presented group and $H$ a finitely generated subgroup. If $H$ is realizable, then the subgroup membership problem of $H$ in $G$ is decidable.
\end{thmx}

This result generalizes a theorem by Jeandel~\cite{jeandel2015aperiodic2} (Theorem~\ref{thm:jeandel}), and is inspired by its proof. It implies that some hyperbolic groups and $\F_n\times\F_n$ have non-realizable subgroups (Examples~\ref{ex:Rips} and~\ref{ex:Mihailova}).\\

There are other algorithmic restrictions to realizability for subgroups of $\Z^d$ relating to formal language theory and computational complexity that we explore in Section~\ref{sec:restrictionsZd}.

\subsection*{Periodically Rigid Groups}

A folklore result for $\Z^2$-SFTs states that every weakly aperiodic SFT is also strongly aperiodic (see Lemma~\ref{lem:PR_Z2}). In other words, non-trivial subsets of $\Sub(\Z^2)$ comprised exclusively of infinite index subgroups are not realizable. The natural question that follows is, which groups exhibit this behaviour? To answer this question, we introduce the class of periodically rigid groups. A group $G$ is said to be \define{periodically rigid} if every weakly aperiodic SFT on $G$ is strongly aperiodic.\\

Inspired by a question of Pytheas-Fogg~\cite{pytheas2022conjecture}, we propose the following conjecture characterizing finitely generated periodically rigid groups.
\begin{conjx}
\label{intro:conjetura}
	A finitely generated group is periodically rigid if and only if it is either virtually $\Z$ or torsion-free virtually $\Z^2$.
\end{conjx}

This conjecture has been shown to hold for many classes of groups (without the explicit intention to do so), such as Baumslag-Solitar groups~\cite{esnay2022weakly} and hyperbolic groups~\cite{coornaert2006symbolic,gromov1987hyperbolic} (see Section~\ref{sec:PeriodicRigidity}).\\

We provide many inheritance properties for the class of periodically rigid groups, and prove that the conjecture holds for virtually nilpotent and polycyclic groups.
\begin{thmx}
\label{intro:vnilp}
    Finitely generated virtually nilpotent groups are periodically rigid if and only if they are not virtually $\Z$ or torsion-free virtually $\Z^2$.
\end{thmx}

\begin{thmx}
\label{intro:poly}
	Finitely generated polycyclic groups are periodically rigid if and only if they are not virtually $\Z$ or torsion-free virtually $\Z^2$.
\end{thmx}
 \subsection*{Canonical Constructions}
To obtain results on realizability, we will make extensive use of subshift constructions that allow us to move subshifts between groups. These constructions are the following:
\begin{itemize}
    \item Free extension: this construction lifts a subshift from a subgroup to its overgroup, preserving non-emptiness and SFTs. We will recurrently use the description of its stabilizer by Barbieri~\cite{barbieri2023aperiodic_non_fg} (see Theorem~\ref{thm:freeX}).
    \item Higher power: this construction takes a subshift from a group to a finite index subgroup. It is a generalization of the classic higher power construction from one-dimensional symbolic dynamics~\cite{lind2021introduction}.
    \item Pull-back: this construction moves a subshift from a quotient to the original group. It preserves non-emptyness, and being an SFT provided the kernel of the quotient is finitely generated. We describe its stabilizer in Proposition~\ref{prop:quotient}. 
    \item Push-forward: this construction moves a subshift from a group to a quotient, provided the subshift is stabilized by the kernel of the projection. Once again, this construction preserves non-emptiness, and preserves SFTs for particular generating sets provided the kernel is finitely generated. We describe its stabilizer in Proposition~\ref{prop:vuelta}. 
\end{itemize}
We provide basic properties of these constructions and short accounts of the results for which they have been used in the literature.
\paragraph{Structure of the Article}

We start with Section~\ref{sec:background}, where we provide the necessary background on symbolic dynamics and group theory. We also provide a short survey on the problems of the existence of strongly aperiodic SFTs and weakly aperiodic SFTs.  Next, Section~\ref{sec:canonical} introduces the four canonical constructions we use throughout the article. The study of realizability begins in earnest in Section~\ref{sec:realizability}, where we provide general properties and prove Theorem~\ref{intro:ida_y_vuelta}. We study the connections between realizability and computability in Section~\ref{sec:computational}. In Section~\ref{sec:PeriodicRigidity}, we introduce and study periodically rigid groups. Finally, in Section~\ref{subsec:nilp_poly}, we prove Conjecture~\ref{intro:conjetura} holds for virtually nilpotent groups and polycyclic groups.

\paragraph{Conventions and Notation} Throughout the article we suppose groups are infinite, unless explicitly stated. The empty word is denoted by $\epsilon$. Finite subsets are denoted by $F\Subset G$. We denote the free group defined by the free generating set of size $n$ by $\F_n$, and $\F_S$ the free group generated by $S$.

\section{Background and Definitions}
\label{sec:background}

\subsection{Symbolic Dynamics}
\label{sec:symb}

Let $G$ be a finitely generated group, and $A$ a non-empty finite set which we call the \define{alphabet}. Elements of $A$ are referred to as \define{letters}, \define{symbols} or \define{tiles} depending of the context. The space of \define{configurations} or \define{tilings} with alphabet $A$ over $G$ is the set $A^G = \{x\colon G\to A\}$. This space is endowed with a left group action $G\act A^{G}$ given by
$$(g\cdot x)(h) = x(g^{-1}h),$$
for all $x\in A^G$ and $h\in G$. This action is referred to as the \define{shift}. The dynamical system $(A^G, G)$ is called the \define{full $G$-shift} over $A$. \\


A \define{pattern} is a map $p\in A^{F}$, where $F$ is a finite subset of $G$ called the \define{support} of $p$. We denote this by $\supp(p) = F$.  We denote the set of all patterns by $A^{\ast G}$. We say a pattern $p$ \define{appears in a configuration} $x\in A^G$, denoted $p\factor x$, if there exists $g\in G$ such that $p(h) = x(gh)$ for all $h\in \supp(p)$. 

Given a set of patterns $\Fo$, we define the $G$-\define{subshift} $\X^G_{\Fo}$ as the set of configurations where no pattern from $\Fo$ appears. That is,
$$\X^G_{\mathcal{F}} = \{x\in A^G \mid \forall p\factor x,\  p \notin\Fo \}.$$

When the group is clear from context we drop the superscript $G$ and write $\X_{\Fo}$ for the $G$-subshift generated by $\Fo$.
A subshift $X\subseteq A^G$ is a $G$-\define{subshift of finite type} ($G$-SFT) if there exists a finite set of forbidden patterns $\Fo$ such that $X = \X_{\Fo}$. We will simply write SFT when the group is clear from context.\\

Let $S$ be a finite generating set for $G$. We say a pattern $p$ is \define{nearest neighbor} if its support is given by $\{1_G, s\}$ for some $s\in S$. We denote nearest neighbor patterns through tuples $(a,b,s)$ representing $p(1_G) = a$ and $p(s) = b$. A subshift defined by a set of nearest neighbor forbidden patterns is known as a \define{nearest neighbor subshift}. These subshifts are necessarily SFTs, as the maximal number of such patterns is bounded by $|A|^2\cdot|S|$. Nearest neighbor SFTs will help us simply the task of searching for an SFT that realizes a given family of subgroups. 

\begin{definition}
    Two $G$-subshifts $X\subseteq A^{G}$ and $Y\subseteq B^G$ are \define{conjugate} if there exists a bijective continuous $G$-equivariant map $\phi:X\to Y$.
\end{definition}

\begin{lemma}
	For $G$ a finitely generated group with finite generating set $S$, every SFT is conjugate to a nearest neighbor SFT with respect to $S$.
\end{lemma}
A proof of this lemma can be found in~\cite{aubrun2018domino}.\\

The \define{orbit} of a configuration $x\in A^G$ is the set of configurations 
\[
\orb(x)=\left\{g\cdot x \mid g\in G \right\}
\]
and its \define{stabilizer} is the subgroup
\[
\stab(x)=\left\{g\in G \mid g\cdot x = x \right\}.
\]

In words, the stabilizer of a configuration $x$ is the set of elements whose action on $x$ leave $x$ unchanged. Notice that the stabilizer is a subgroup of $G$. Elements of the stabilizer are called \define{periods} of the configuration. We define the set of stabilizers of a $G$-subshift $X$ as 
$$\stab(X) = \{\stab(x) \mid x\in X\}.$$
We say a configuration $x\in A^G$ is \define{periodic} if $\stab(x)$ is a finite index subgroup, and \define{aperiodic} if $\stab(x) = \{1_G\}$.
\begin{remark}
	In the literature periodic configurations are sometimes called strongly periodic, and configurations with non-trivial stabilizers are called weakly periodic. We do not use this terminology.
\end{remark}

\begin{definition}
	We say a subshift $X\subseteq A^G$ is
	\begin{itemize}
		\item \define{weakly aperiodic} if every configuration has an infinite orbit. Equivalently, if $X$ contains no periodic points.\index{subshift! weakly aperiodic}
		\item \define{strongly aperiodic} if the stabilizer of every configuration is trivial. \index{subshift! strongly aperiodic}
	\end{itemize}
\end{definition}

\begin{remark}
	Although the empty set is both a strongly and weakly aperiodic subshift, for the purposes of this work we \underline{do not} consider the empty subshift when talking about aperiodicity, unless explicitly stated.
\end{remark}

\begin{definition}
    The \define{kernel} of a subshift $X\subseteq A^G$ is 
    $$\ker(X) = \bigcap_{x\in X}\stab(x).$$
    Equivalently, $g\in\ker(X)$ if $g\cdot x = x$ for all $x\in X$.
\end{definition}

\subsection{Group presentations and the word problem}
\label{sec:groups}

Let $G$ be a finitely generated group and $S$ a finite generating set. We will only consider finite \define{symmetric} generating sets, that is, generating sets such that $S=S^{-1}$. Elements in the group are represented as words on the alphabet $S$ through the evaluation function $w\mapsto \overline{w}$. Two words $w$ and $v$ represent the same element in $G$ when $\overline{w} = \overline{v}$, and we denote this by $w =_G v$. We denote the identity of a group $G$ by $1_G$.

\begin{definition}
	Let $G$ be a group. We say $(S, R)$ is a \define{presentation} of $G$, denoted $G = \langle S \mid R\rangle$, if the group is isomorphic to  $\langle S \mid R\rangle = \F_S/\llangle R\rrangle$, where $\llangle R\rrangle$ is the normal closure of $R$, i.e. the smallest normal subgroup of $\F_S$ containing $R$.
\end{definition}

 We say $G$ is \define{finitely presented} if it has a presentation $(S,R)$ where $S$ and $R$ are finite, and \define{recursively presented} if there exists a presentation $(S,R)$ such that $S$ is recursive and $R$ is recursively enumerable.
 
For a group $G$ and a generating set $S$, we define:
  $$\WP(G, S) = \{w\in S^{*} \ | \ w=_{G} \epsilon\}.$$

 \begin{definition}
  The \define{word problem} of a group $G$ with respect to a set of generators $S$ is the following decision problem: given a word $w\in S^*$, determine whether $w\in \WP(G,S)$.
 \end{definition}

\begin{remark}
A simple computation shows that for two different generating sets of $G$, $S_1$ and $S_2$, the word problem with respect to $S_1$ is many-one equivalent to the word problem with respect to $S_2$. We can therefore talk about the word problem of the \emph{group}, which we denote by $\WP(G)$. We also denote by co$\WP(G)$ the complement of the word problem.
\end{remark}

A key connection between the presentation of a group and its word problem is the following.
\begin{proposition}
\label{prop:WP}
	Let $G$ be a finitely generated group. $G$ is recursively presented if and only if $\WP(G)$ is recursively enumerable.
\end{proposition}

An element $g\in G$ has \define{torsion} if there exists $n\geq 1$ such that $g^n = 1_G$. If there is no such $n$, we say $g$ is \define{torsion-free}. Analogously, we say $G$ is a \define{torsion group} if all of its elements have torsion. Otherwise, if all non-trivial elements of $G$ are torsion-free, we say the group is \define{torsion-free}.\\ 

Finally, let $\mathcal{P}$ be a class of groups. We say a group $G$ is \define{virtually} $\mathcal{P}$, if there exists a finite index subgroup $H\leq G$ that is in $\mathcal{P}$.

\subsection{Strongly aperiodic SFTs}
\label{sec:strongly}

Let us briefly look at the current state of the existence of strongly and weakly aperiodic SFTs, starting by the former. The first example of a strongly aperiodic SFT for $\Z^2$ was famously constructed by Berger~\cite{berger1966undecidability}. Since then, many other aperiodic SFTs for $\Z^2$ have been constructed~\cite{robinson1971undecidability,kari1996small,jeandel2021wang}. Berger's result also motivated the study of aperiodic subshifts on groups beyond $\Z^2$, which we review in this section.\\

There are several structural and algorithmic necessary conditions that a group must satisfy in order to allow a strongly aperiodic SFT. The first of these is due to Jeandel, and relates aperiodicity to the word problem of the group.

\begin{theorem}[Jeandel~\cite{jeandel2015aperiodic2}]
	\label{thm:jeandel}
	Let $G$ be a finitely generated group. If $G$ admits a strongly aperiodic SFT, then $\WP(G)\leq_{e}\co\WP(G)$. In particular, if $G$ is recursively presented it has decidable word problem.
\end{theorem}

The relation $A\leq_e B$ between two languages $A$ and $B$, represents the fact that $A$ \define{enumeration-reduces} to $B$. Intuitively, this means that it is possible to compute an enumeration for $A$ given an enumeration for $B$. A set satisfying $A\leq_{e}\co A$ is called co-total. Thus, Jeandel's theorem states that the existence of strongly aperiodic SFTs implies that the word problem of the group is co-total.\\

The next restriction is due to Cohen, who related the existence of strongly aperiodic SFTs to the large scale structure of the group, specifically the amount of ends.

\begin{theorem}[Cohen~\cite{cohen2017large}]
	\label{thm:CohenEnds}
	If $G$ is a finitely generated group with at least two ends, then it does not admit a strongly aperiodic SFT.
\end{theorem}

By combining Jeandel and Cohen's results, we arrive at the following conjecture.
\begin{conjecture}
	\label{conj:SA}
	Let $G$ be a finitely generated group. $G$ admits a strongly aperiodic SFT if and only if $G$ is one ended and has decidable word problem.
\end{conjecture}

Notice however that Jeandel's Theorem does not rule out the existence of groups with undecidable word problem that satisfy $\WP(G)\leq_e\textnormal{co}\WP(G)$ and admit a strongly aperiodic SFT. In fact, some researchers believe such groups exist\footnote{Sebasti{\'a}n Barbieri, personal communication.}.

So far, the conjecture has been shown to hold for the following classes of finitely generated groups:

\begin{itemize}
	\item Virtually polycyclic groups~\cite{jeandel2015aperiodic},
	\item Baumslag-Solitar groups~\cite{aubrun2023gbs}, with the solvable case originally obtained in~\cite{esnay2022weakly} with an alternative proof in~\cite{Aubrun_Schraudner_2020},
	\item The Ivanov monster group, for which every element is cyclic and contains a finite number of conjugacy classes, and the Osin monster groups, which contains two conjugacy classes~\cite{jeandel2015aperiodic2},
	\item Surface groups \cite{cohen2017strongly}, and more generally hyperbolic groups~\cite{cohen2022strongly},
	\item Groups of the form $\Z^2\ltimes_{\phi} H$ where $H$ has decidable word problem~\cite{barbieri2019simulation}. An indepent proof also exists for the particular case of the Heisenberg group~\cite{Sahin_Schraudner_Ugarcovici_2021},
	\item Groups of the form $G\times H\times K$ where each group has decidable word problem~\cite{barbieri2019geometric}. This includes finitely generated branch groups with decidable word problem such as the Grigorchuk group,
	\item Self-simulable groups with decidable word problem~\cite{barbieri2021groups}. Self-simulable groups include the direct product of any two non-amenable groups as well as Thompson's group $V$, Burger-Mozes simple finitely presented group, braid groups on more than $7$ stands, some RAAGs, among others,
	\item Groups of the form $H\times N$ where both groups have decidable word problem and $N$ is non-amenable~\cite{barbieri2023soficity}. This includes some groups who where previously known to admit strongly aperiodic SFTs such as $\Z\times V$, $\Z\times T$ and $\Z\times PSL_2(\Z)$ where $V$ and $T$ are Thompson's groups~\cite{jeandel2015aperiodic2},
	\item the Lamplighter group~\cite{bartholdi2024shifts}.
\end{itemize}

A particularly important property of strong aperiodicity is that it is a geometric property for finitely presented groups.
\begin{theorem}[Cohen \cite{cohen2017large}]\index{quasi-isometry}
\label{thm:CohenQI}
	Let $G$ and $H$ be two quasi-isometric finitely presented groups. Then, $G$ admits a strongly aperiodic SFT if and only if $H$ does.
\end{theorem}


A similar invariance result for finitely generated groups has been obtained for commensurable groups.

\begin{theorem}[Carroll, Penland \cite{carroll2015periodic}]
	Let $G$ and $H$ be two finitely generated groups which are commensurable. Then, $G$ admits a strongly aperiodic SFT if and only if $H$ does.
\end{theorem}

To finish this section, let us comment on what happens when one alleviates the restrictions of finite type or finite generation. Gao, Jackson and Seward showed that every countable group admits a strongly aperiodic subshift~\cite{gao2009coloring}. This was later improved upon by Aubrun, Barbieri and Thomassé who in addition to finding an alternative proof for the result using the Lov\'asz Local Lemma, showed that when the group is recursively presented the constructed strongly aperiodic subshift is effectively closed~\cite{aubrun2019realization}. 

On the side of non-finitely generated groups, Barbieri characterized groups that admit strongly aperiodic SFT in terms of their finitely generated subgroups.

\begin{theorem}[Barbieri~\cite{barbieri2023aperiodic_non_fg}]
	A group $G$ admits a strongly aperiodic SFT if and only if there exists a finitely generated subgroup $H\leq G$ and a non-empty $H$-SFT $X$ such that for every $g\in G\setminus\{1_G\}$ there exists $t\in G$ and $n\in\N$ such that $tg^nt^{-1}\in H\setminus\bigcup_{x\in X}\stab(x)$.
\end{theorem}

The strongly aperiodic shift comes from taking the free extension of $X$ onto $G$ (see Definition~\ref{def:free_extension}). This allowed Barbieri to show that groups such as $\mathbb{Q}^2$ admit strongly aperiodic SFT, and to find an alternative proof for the existence of such SFTs on the Osin monster group.

\subsection{Weakly aperiodic SFTs}
\label{sec:weakly}

Let us now review weakly aperiodic SFTs. Our first observation is that for infinite groups, any strongly aperiodic SFT is weakly aperiodic, as the orbit of any configuration is in bijection to the quotient of the group by the corresponding stabilizer. This already gives us a number of examples of groups that admit weakly aperiodic SFTs. Nevertheless, there are groups that admit weakly aperiodic SFTs, but not strongly aperiodic ones, such as free groups (by Theorem~\ref{thm:CohenEnds}). Carroll and Penland showed that the only virtually nilpotent groups that do not admit weakly aperiodic SFTs are virtually $\Z$ groups~\cite{carroll2015periodic}. This motivated them to propose the following conjecture.

\begin{conjecture}
	\label{conj:CarrollPenland}
	Let $G$ be a finitely generated group. Then, $G$ admits a weakly aperiodic SFT if and only if it is not virtually $\Z$.
\end{conjecture}


At the time of writing, the following classes of finitely generated groups have been shown to satisfy this conjecture:
\begin{itemize}
	\item Virtually nilpotent groups~\cite{ballier2018domino,carroll2015periodic}, and more generally virtually polycyclic groups~\cite{jeandel2015aperiodic},
	\item Baumslag-Solitar groups~\cite{aubrun2013baumslag}, and more generally generalized Baumslag-Solitar groups~\cite{aubrun2023gbs}
	\item Hyperbolic groups~\cite{coornaert2006symbolic,gromov1987hyperbolic},
	\item Non amenable groups \cite{jeandel2015translation,block1992aperiodic},
	\item Non residually finite groups \cite{jeandel2015translation},
    \item Artin groups~\cite{aubrun2023gbs},
	\item Infinite finitely generated $p$-groups~\cite{jeandel2015translation,marcinkowski2014aperiodic},
	\item Groups of the form $G_1\times G_2$ where both groups are infinite~\cite{jeandel2015translation}. This shows the Grigorchuk group admits a weakly aperiodic SFT, which was also obtained in~\cite{marcinkowski2014aperiodic},
	\item The Lamplighter group~\cite{cohen2020lamplighters}.
\end{itemize}

There are also many properties satisfied by groups which do admit weakly aperiodic SFTs. We summarize these properties in the following proposition.
\begin{proposition}
	\label{prop:WA}
	Let $G$ be a finitely generated group. The following hold,
	\begin{itemize}
		\item If $G$ is commensurable to $H$, and $H$ admits a weakly aperiodic SFT, then so does $G$~\cite{carroll2015periodic},\index{commensurability}
		\item If a subgroup $H\leq G$ admits a weakly aperiodic SFT, then so does $G$ (see Lemma~\ref{lem:free_extension}),
		\item For a finitely generated normal subgroup $N\trianglelefteq G$, if $G/N$ admits a weakly aperiodic SFT, then so does $G$ (see Proposition~\ref{prop:quotient}),
		\item If a finitely presented group $H$ acts translation-like on $G$, and $H$ admits a weakly aperiodic SFT, then so does $G$~\cite{jeandel2015translation}.\index{translation-like action}
	\end{itemize}
	Furthermore, if $G$ does not admit a weakly aperiodic SFT,
	\begin{itemize}
		\item For every $n\in\N$ there must exist a finite index subgroup $H\leq G$ such that $n$ divides $[G:H]$~\cite[Corollary 3.3]{jeandel2015translation} (see also~\cite{marcinkowski2014aperiodic}),
		\item If $G$ is finitely presented, then it must contain a finite index subgroup $H$ that surjects into $\Z$~\cite{cohen2020lamplighters} (groups with this property are called \define{virtually indicable}),
		\item If $G$ is finitely presented and is quasi-isometric to a finite presented group $H$, then there must exist $G_0\leq G$ and $H_0\leq H$ finite index subgroups such that $H_0$ is isomorphic to a quotient of $G_0$ by a finite subgroup~\cite{cohen2017large}.
	\end{itemize}
\end{proposition}

\section{Canonical constructions}
\label{sec:canonical}

This section is focused on introducing constructions that allow us to go from a subshift defined over a group, to one defined on a subgroup or quotient. These constructions appear under varied names throughout the literature, and have been used to prove many results. We will point out their uses as we introduce them.

\subsection{Free extension}

The first of the constructions we look at is the \define{free extension}. This construction is perhaps the most direct way to define a subshift on a group starting from a subshift on one of its subgroups. It is because of this that it appears often in the literature, though not always under the same name. Free extensions have been used by Hochman and Meyerovitch~\cite{hochman2010entropies} to characterize the entropies of $\Z^d$-SFTs, by Ballier and Stein for the Domino Problem~\cite{ballier2018domino}, by Jeandel to prove results about strongly and weakly aperiodic SFTs and the Domino Problem~\cite{jeandel2015aperiodic,jeandel2015translation}, by Carrol and Penland to prove the commensurability invariance of aperiodicity~\cite{carroll2015periodic}, by Barbieri to study the set of possible entropies of certain amenable groups~\cite{barbieri2021entropies}, and by Barbieri, Sablik and Salo to obtain simulation theorems on direct products of groups~\cite{barbieri2023soficity}. Recently Raymond has made a careful study of the properties of free extensions in order to study SFTs on locally finite groups~\cite{raymond2023shifts}. 
\begin{definition}
	\label{def:free_extension}
	Let $G$ be a group, $H\leq G$ a subgroup, $X\subseteq A^H$ an $H$-subshift. The \define{free extension} of $X$ to $G$, denoted $X^{\uparrow}$, is the $G$-subshift defined as, 
$$X^{\uparrow} = \{x\in A^G \mid \forall g\in G, \ x|_{gH}\in X\}.$$
\end{definition}

\begin{remark}
	\label{rem:periodic}
There is an alternative way of lifting a subshift over a subgroup to the whole group, which also appears regularly in the literature. Given an $H$-subshift $X\subseteq A^{H}$, and a set of left coset representatives $L$, the \define{periodic (or trivial) extension} of $X$ is the subshift
$$X^{\Uparrow} = \{y\in A^G \mid \exists x\in X, \forall l\in L, \ y|_{lH} = x\}\subseteq\upa{X}.$$
This extension famously appears on the simulation results independently discovered by Hochman~\cite{hochman2009effective}, Aubrun and Sablik~\cite{aubrun2013simulation}, and Durand, Romaschenko and Shen~\cite{durand2012fixed}. It has also been used to prove a Higman Embedding Theorem for subshifts~\cite{jeandel2019characterization}, characterize extender entropies of $\Z^d$-subshifts~\cite{callard2024computability}, and for simulation theorems on other finitely generated groups~\cite{barbieri2019geometric,barbieri2019simulation,barbieri2021groups}.
\end{remark}

\begin{lemma}
\label{lem:free_extension}
	Let $X\subseteq A^{H}$ be an $H$-subshift generated by the set of forbidden patterns $\Fo$, that is, $X = \X^{H}_{\Fo}$. Then, the free extension satisfies the following
	 \begin{enumerate}
	 	\item $\upa{X} = \X^{G}_{\Fo}$,
	 	\item $\upa{X}$ is empty if and only if $X$ is empty,
	 	\item for every $x\in\upa{X}$, $\stab(x)\cap H \subseteq \stab(x|_H)$.
	 \end{enumerate}
\end{lemma}

\begin{proof}
	\begin{enumerate}
		\item Take $x\in\upa{X}$, and suppose there is $g\in G$ and $p\in\Fo$ such that $x|_{g\cdot\supp(p)} = p$. Then, the configuration $y = x|_{gH}$ contains the forbidden pattern $p$, which contradicts the fact that $y\in X$. Conversely, take $x\in\X_{\Fo}^G$ and $g\in G$. Because the supports of patterns from $\Fo$ are contained in $H$, $x|_{gH}\in \X_{\Fo}^{H} = X$. Thus, $x\in\upa{X}$.
		
		\item For any $x\in\upa{X}$, we have $x|_{H}\in X$. Take $L$ a set of left coset representatives for $H$. Given a configuration $y\in X$, we define $x\in A^{G}$ by $x(lh) = y(h)$ for all $l\in L$ and $h\in H$. Then, given $g\in G$ there exists $l\in L$ and $h\in H$ such that $g = lh$. Then, $x(gh') = x(lhh')  = y(hh')$, for all $h'\in H$. Thus, $x|_{gH} = h^{-1}\cdot y\in X$.
		
		\item Take $x\in\upa{X}$, $h\in\stab(x)\cap H$, and $y = x|_{H}$. Then,
		$$h\cdot y(h') = y(h^{-1}h') = x(h^{-1}h') = x(h') = y(h').$$
		Therefore, $h\in\stab(y)$.
	\end{enumerate}
\end{proof}

A more manageable way to understand free extensions is through the cosets of the subgroup. 
\begin{lemma}
	Take $L$ a set of left coset representatives for $G/H$ and an $H$-subshift $X$. Then, $y\in\upa{X}$ if and only if there exist a collection of configurations from $X$, $(x_{l})_{l\in L}$, such that $y|_{lH} = x_l$.
\end{lemma}

\begin{proof}
	For a configuration $y\in\upa{X}$, we define our collection of configurations as $x_l = y|_{lH}\in X$. Conversely, let $y$ be defined by the collection $\{x_{l}\}_{l\in L}$. Take $g\in G$, which is uniquely written as $g = lh$ for some $l\in L, h\in H$. Then, $y(gh') = y(lhh') = x_l(hh')$ for all $h'\in H$, implying that $y|_{gH} = h^{-1}\cdot x_l\in X$. Thus, $y\in\upa{X}$.
\end{proof}





\subsection{Higher power and higher block}
\label{subsec:higher}

Higher block and higher power subshfits are standard and useful constructions for $\Z$-subshifts. They allow both to re-scale subshifts -- in order to find letter-to-letter sliding-block codes and nearest neighbor subshifts-- and to go from a group to a finite index subgroup. This construction has been used by Carroll and Penland to prove aperiodicity is a commensurability invariant~\cite{carroll2015periodic}, and by Aubrun, Barbieri and Jeandel to prove every SFT is conjugate to a nearest neighbor SFT~\cite{aubrun2018domino}.

\begin{definition}
	\label{def:higher_power}
	Let $G$ be a group, $H\leq G$ a finite index subgroup, and $X\subseteq A^{G}$. Take $R$ a set of right coset representatives. The $R$-\define{higher power} of $X$, denoted $X^{[R]}$, is the $H$-subshift over the alphabet $A^R$ defined as, 
	$$X^{[R]} = \left\{x\in (A^R)^H \mid \exists y\in X,\  \forall(h,r)\in H\times R,\ x(h)(r) = y(hr)\right\}.$$
\end{definition}

\begin{lemma}
	Let $X\subseteq A^G$ be a $G$-subshift, $H\leq G$ a finite index subgroup, and $R$ a set of right coset representatives. Then, $X$ is non-empty if and only if $X^{[R]}$ is non-empty.
\end{lemma}

\begin{proof}
	For a configuration $y\in X$ we define $x\in A^{H}$ by $x(h) = y|_{hR}$. By definition, $x\in X^{[R]}$. Conversely, for a configuration $x'\in X^{[R]}$ there exists $y'\in X$ such that $x(h)(r) = y(hr)$ for all $h\in H$ and $r\in R$.
\end{proof}

Given a set of forbidden patterns $\Fo$ for the base $G$-subshift $X$, we create a set of forbidden patterns, $\Fo'$, for $X^{[R]}$. Take $q\in\Fo$ with support $F\Subset G$, and $r\in R$. For each $f\in F$ there exists $h_f\in H\cap RFR^{-1}$ and $r_f\in R$ such that $h_fr_f = rf$. Define the set $P(q,r) = \{h_{f} \mid f\in F\}$. Now, for each $q\in Q$ and $r\in R$ the set $\Fo'$ contains all the patterns $p$ of support $P(q,r)$ such that $p(h_f)(r_f) = q(f)$.

\begin{lemma}
	\label{lem:patterns_higher_block}
	Let $X\subseteq A^{G}$ be an $G$-subshift given by the set of forbidden patterns $\Fo$. Then $X^{[R]} = \X_{\Fo'}^H$.
\end{lemma}

\begin{proof}
	Take $x\in X^{[R]}$. By definition there is a configuration $y\in X$ such that $x(h)(r) = y(hr)$ for all $h\in H$ and $r\in R$. Suppose $x\notin\X_{\Fo'}$. Then, there exists $h\in H$, $q\in\Fo$, $r\in R$ and $p\in\Fo'$ of support $P(q,r)$ such that $x|_{hP(q,r)} = p$. For each $f\in\supp(q)$,
	$$y(hrf) = y(hh_fr_f) = x(hh_f)(r_f) = q(f).$$
	In other words, $y|_{hr\cdot\supp(q)} = q$, which is a contradiction.
	
	Conversely, take $x\in\X_{\Fo'}$ and define $y\in A^G$ as $y(hr) = x(h)(r)$ for all $h\in H$ and $r\in R$. Suppose $y\notin X$. Then, there exists $h\in H$, $r\in R$ and $q\in\Fo$ such that $y|_{hr\cdot\supp(q)} = q$. Then, for all $h_f\in P(q,r)$
	$$x(hh_f)(r_f) = y(hh_fr_f) = y(hrf) = q(f).$$
	This means $x|_{hP(q,r)}\in\Fo'$, which is a contradiction. Therefore $y\in X$ and $x\in X^{[R]}$.
\end{proof}

A slight generalization of the $R$-higher power shift has also been used in the literature (\cite{carroll2015periodic} for instance). This generalization, called the \define{$R$-higher block shift}, consists in taking any finite set $R\Subset G$ such that $HR = G$ that is not necessarily a set of coset representatives. Except for this change, the definition is the same. 



\subsection{Pull-back}

Our next construction allows us to go from a quotient to a group. We call it the \define{pull-back shift}. This construction has been used by Ballier and Stein to show that the undecidability of the Domino Problem and weak aperiodicity can be transported from the quotient to the group~\cite{ballier2018domino}, by Jeandel to construct strongly aperiodic SFTs on polycyclic groups~\cite{jeandel2015aperiodic}, and by Bartholdi and Salo to obtain simulation theorems for the Lamplighter group~\cite{bartholdi2024shifts}. Furthermore, the last two authors argue that the simulation results mentioned in Remark~\ref{rem:periodic} are in fact results about pull-backs. This is restated by Grigorchuk and Salo in~\cite{grigorchuk2024sft}.

\begin{definition}
	Let $G$ be a group, $N\trianglelefteq G$ a normal subgroup, and $X\subseteq A^{G/N}$ an $G/N$-subshift. The \define{pull-back} of $X$ to $G$, denoted $\pi^*(X)$, is the $G$-subshift defined as, 
	$$\pi^*(X) = \{x\circ\pi\in A^G \mid x\in X\},$$
	where $\pi:G\to G/N$ is the quotient map.
\end{definition}

\begin{remark}
	Notice that $\pi^*(X)$ is non-empty if and only if $X$ is non-empty.
\end{remark}
 
Now, let us build a set of forbidden patterns $\Fo'$ for $\pi^*(X)$ from a set of forbidden patterns $\Fo$ for $X$. Let $\rho:G/N\to G$ be a section and $T$ a set of generators for $N$. The set $\Fo'$ contains for each $t\in T$ all patterns $q:\{1_G,t\}\to A$ such that $q(1_G)\neq q(t)$, and for each $p\in\Fo$ a pattern $q:\rho(\supp(p))\to A$ defined by $q(\rho(g)) = p(g)$.

\begin{lemma}
	\label{lem:patterns_pull_back}
	Let $X\subseteq A^{G/N}$ be a $G/N$-subshift given by the set of forbidden patterns $\Fo$. Then, $\pi^*(X) = \X_{\Fo'}^G$. In particular, if $N$ is finitely generated and $X$ is an SFT, $\pi^*(X)$ is an SFT.
\end{lemma}
 
\begin{proof}
	Take $y\in\pi^*(X)$ and $x\in X$ such that $y = x\circ\pi$. Suppose $y\notin\X_{\Fo'}$. If there are distinct $a,b\in A$, $g\in G$ and $t\in T$, such that $y(g) = a$ and $y(gt) = b$, this means $x(\pi(gt)) = x(\pi(g)) = a$ and $x(\pi(g)) = b$ which is a contradiction. On the other hand, if there is $g\in G$ and $p\in\Fo$ such that $y|_{g\cdot\rho(\supp(p))} = p$, then $x|_{\pi(g)\cdot\supp(p)} = p$ which is a contradiction. Therefore, $y\in\X_{\Fo'}$. 
	
	Conversely, take $y\in\X_{\Fo'}$ and define $x = y\circ\rho\in A^{G/N}$. Now, for every $g\in G$ there exists $h\in N$ such that $\rho(\pi(g)) = gh$. Then, $(x\circ\pi)(g) = y(gh) = y(g)$, as $y$ is invariant on $N$. Finally, if $x\notin X$ because it has a pattern $p\in\Fo$, $y$ would have the pattern $p\circ\rho\in\Fo'$ which is a contradiction. Thus $x\in X$ and $y\in\pi^*(X)$.
\end{proof}

\subsection{Push-forward}
\label{subsec:push_forward}

Our final construction is the \define{push-forward shift}. So far this construction has not been present in the literature, although it does fall under the definition of pull-back shifts as defined by Bartholdi and Salo~\cite{bartholdi2024shifts}.\\

To transport a subshift from a group to its quotient we must ask additional properties on the subshift. To do this we introduce the $N$-fixed subshift. Let $A$ be a finite alphabet, $G$ a finitely generated group with $N$ a normal subgroup. We define the subshift
$$\text{Fix}_A(N) = \{x\in A^{G}\mid n\cdot x = x, \ \forall n\in N\}.$$

\begin{remark}
	For any subgroup $N$, $\text{Fix}_A(N)$ is always a closed set of $A^G$, but it is only shift invariant when $N$ is normal. In the latter case, the subshift is conjugate to $A^{G/N}$.
\end{remark}

Notice when $N$ is finitely generated, $\Fix_A(N)$ is an SFT. If we take $S = \{s_1, ..., s_m\}$ a set of symmetric generators for $N$, we see that $\text{Fix}_R(N)$ is the SFT by the set of forbidden rules given by 
$$\{p:\{1,s_i\}\to R \mid s_i\in S, \ p(1)\neq p(s_i)\}.$$

\begin{remark}
    A subshift $X\subseteq A^G$ is contained in $\Fix_A(N)$ if and only if $N\leq \ker(X)$.
\end{remark}

\begin{definition}
	Let $G$ be a group, $N\trianglelefteq G$ a finitely generated normal subgroup, $X\subseteq\Fix_A(N)$ a $G$-subshift, and $\rho:G/N\to G$ a section. The \define{push-forward} of $X$ to $G/N$, denoted $\rho^*(X)$, is the $G$-subshift defined as, 
	$$\rho^*(X) = \{x\circ\rho\in A^{G/N} \mid x\in X \}.$$
\end{definition}

As was the case with the pull-back, we have that $\rho^*(X)$ is non-empty if and only if $X$ is non-empty.

\begin{lemma}
	Let $X\subseteq\Fix_A(N)$ be a $G$-subshift. Then, the push-forward is independent of the section, that is, for any two sections $\rho_1,\rho_2:G/N\to G$, $\rho_1^*(X) = \rho_2^*(X)$.
\end{lemma}

\begin{proof}
	Take $x\in\rho_1^*(X)$. By definition there is a configuration $y\in X$ such that $x = y\circ\rho_1$. Because $\rho_1$ and $\rho_2$ are sections, for all $g\in G/N$ there exists some $h\in N$ such that $\rho_1(g) = h\rho_2(g)$. Given that $y\in\Fix_A(N)$, we have that 
	$$x(g) = y(\rho_1(g)) = y(h\rho_2(g)) = y(\rho_2(g)),$$
	for all $g\in G/N$. Thus, $x = y\circ \rho_2$ and therefore belongs in $\rho_2^*(X)$. Because the previous argument is independent of the section, we conclude that $\rho_1^*(X) = \rho_2^*(X)$.
\end{proof}

Because of this result, we talk about \emph{the} push-forward of a subshift $X\subseteq\Fix_A(N)$.

\begin{lemma}
	\label{lem:patterns_push_forward}
	Let $G$ and $N$ be finitely generated, and take $S$ and $T$ finite generating sets for $G/N$ and $N$ respectively. If $X\subseteq A^{G}$ is a nearest neighbor $G$-SFT with respect to $\rho(S)\cup T$, then $\rho^*(X)$ is a nearest neighbor $G/N$-SFT with respect to $S$.
\end{lemma}

\begin{proof}
	Let $\Fo$ be a set of nearest neighbor forbidden patterns with respect to $\rho(S)\cup T$, for $X$. Because $X\subseteq\Fix_A(N)$, we suppose that $\Fo$ contains all patterns $(a,b,t)$ for $a,b\in A$, $a\neq b$ and $t\in T$. Then, define the set of forbidden patterns $\Fo'$ that contains $(a,b,s)$ with $a,b\in A$ and $s\in S$, for each pattern $(a,b,\rho(s))\in\Fo$. Let us show that $\rho^*(X) = \X_{\Fo'}^{G/N}$. Take $x\in\rho^*(X)$ and suppose $x\notin\X_{\Fo'}$. Then, there exist $g\in G/N$ and $s\in S$ such that $(x(g), x(gs), s)\in\Fo'$. Because $x$ belongs to the push-forward, there exists $y\in X$ such that $x = y\circ\rho$. Then, because $y\in\Fix_A(N)$, $x(gs) = y(\rho(g)\rho(s))$. Therefore, $(y(\rho(g)), y(\rho(g)\rho(s)), \rho(s))\in\Fo$ appears in $y$, which is a contradiction.
	
	Conversely, take $x\in\X_{\Fo'}$ and define $y = x\circ\pi\in A^G$. Notice that for any $g\in G$ and $h\in N$,
	$$(h\cdot y)(g) = x(\pi(h^{-1}g))=x(\pi(g)) = y(g).$$
	Next, if $y\notin X$ because there is $g\in G$ and $s\in S$ such that $(y(g), y(g\rho(s)), \rho(s))\in\Fo$, then the pattern $(x(\pi(g)), x(\pi(g)s), s)\in\Fo'$ appears in $x$, which is a contradiction. Thus, $y\in X$. Finally, for any $g\in G/N$, $y\circ\rho(g) = x(\pi(\rho(g))) = x(g)$. Therefore, $x\in\rho^*(X)$.
\end{proof}

As the previous proof suggests, the push-forward and pull-back are complementary constructions in the following sense.

\begin{lemma}
\label{lem:pullpush}
	Let $X\subseteq\Fix_A(N)$ be a $G$-subshift and $Y$ a $G/N$-subshift. Then, $X = \pi^*(\rho^*(X))$ and $Y = \rho^*(\pi^*(Y))$.
\end{lemma}

\begin{proof}
	We begin with $X\subseteq\Fix_A(N)$. For $x\in\pi^*(\rho^*(X))$, there exists $y\in X$ such that $x = y\circ\rho\circ\pi$. As $\rho(\pi(g))\in gN$ for all $g\in G$, and $y\in\Fix_A(N)$, we have that $x = y\in X$. Next, take $z\in X$ and define $y = z\circ\rho\circ\pi$. By definition, $y\in\pi^*(\rho^*(X))$. As before, because $\rho(\pi(g))\in gN$ for all $g\in G$, and $z\in\Fix_A(N)$, both configurations are the same, meaning $z\in\pi^*(\rho^*(X))$.  
	
	For the second statement, take $y\in\rho^*(\pi^*(Y))$. There exists $x\in Y$ such that $y = x\circ\rho\circ\pi$. Because $\rho(\pi(g)) = g$ for all $g\in G/N$, $x$ and $y$ are equal. Finally, take $y\in Y$ and define $z = y\circ\rho\circ\pi$, which by definition belongs to $\rho^*(\pi^*(Y))$. We conclude as before that $y = z\in\rho^*(\pi^*(Y))$.
\end{proof}

This Lemma allows us to show that all subshifts are the pull-forward of some $F$-subshift, for $F$ a free group.

\begin{proposition}
	Let $X$ be a $G$ subshift. Then, there exists a free group $F$ and a $F$-subshift $Y$ such that $\rho^*(Y) = X$.	
\end{proposition}

\begin{proof}
	Let $S$ be a generating set for $G$ and $F = \F_S$ the free group with $S$ as a free generating set such that we have a quotient map $\pi:F\to G$. By Lemma \ref{lem:pullpush}, the push-back of $X$, $Y = \pi^*(X)$ is a $F$-subshift that verifies $\rho^*(Y) = X$.
\end{proof}

%
\section{Subgroup Realizability: How I Learned to Stop Worrying and Love Periodicity}
\label{sec:realizability}

We now move to the core of this article. How much control do we have over the stabilizers of an SFT? Could we replace the trivial subgroup in strongly aperiodic SFTs with any other subgroup? The aim of this section is to explore this question. We will see that there are both algebraic and computational restrictions to the realizability of subgroups as stabilizers.\\

Let us recall the definition.
\begin{definition}
	We say a family of subgroups $\G\subseteq\Sub(G)$ is \define{realizable} if there exists a non-empty $G$-SFT $X$ such that $\stab(X) = \G$. We say $H\in\Sub(G)$ is realizable if the singleton $\{H\}$ is realizable.
\end{definition}

\begin{question}
	Which subsets of $\Sub(G)$ are realizable?
\end{question}

First off, a simple cardinality argument shows that no group allows for all subsets of $\Sub(G)$ to be realizable: $\mathcal{P}(\Sub(G))$ is uncountable while the number of SFTs is countable. A natural starting point then is the realizability of a single subgroup. This question is non-trivial as the realization of the trivial subgroup is equivalent to finding a strongly aperiodic SFT. We also have a cardinality constraint on this \textit{a priori} simpler case, as there are groups whose space of subgroups is uncountable. Examples of such groups include those where their perfect kernel is a Cantor space, such as $\F_n$ (see~\cite{carderi2022space} and the references therein), or groups that admit allosteric actions such as surface groups or $\Z/2\Z\ast\Z/2\Z\ast\Z/2\Z\ast\Z/2\Z$ (see~\cite{joseph2014allosteric}).\\

Even if we restrict ourselves to groups with a countable amount of subgroups, such as $\Z^d$, the problem is non-trivial as shown in the next lemma.

\begin{lemma}
	\label{lem:OGPR}
	In $\Z^2$, the subgroup $p\Z\times\{0\}$ is not realizable. In particular, for any SFT $X\subseteq A^{\Z^2}$, if $p\Z\times\{0\}\in\stab(X)$, then there exists $0<q\leq |A|^p+1$ such that $p\Z\times q\Z\in\stab(X)$.
\end{lemma}

\begin{proof}
	 Take $X$ a nearest neighbor $\Z^2$-SFT and $x\in X$ such that $\stab(x) = p\Z\times\{0\}$. Denote $w^k = x|_{[0,p-1]\times\{k\}}$. Because of $x$'s periodicity, for all $k\in\Z$, the restriction $x|_{\Z\times\{k\}}$ is the periodic configuration $(w^k)^{\infty}$. Now, by the pigeonhole principle, there exist $k\geq 0$ and $0<q \leq |A|^{p+1}$ such that $w^k = w^{k+q}$. Let $sq:[0,p-1]\times[0,q-1]\to A$ be the rectangular pattern defined by $sq(i,j) = (w^{k+j})_{i}$. Define $y\in A^{\Z^2}$ by $y(i,j) = sq(i\!\mod p, j\!\mod q)$. By construction, $y$ belongs to $X$ and its stabilizer is $p\Z\times q\Z$ and  belongs to $X$.
\end{proof}

We generalize this phenomenon to find more examples of non-realizable subgroups in the subsequent sections.

\subsection{General properties}

We saw in Section~\ref{sec:symb} that every SFT is conjugate to a nearest neighbor SFT. This result allows us to restrict the scope of the SFTs we consider, as stabilizers are preserved under conjugacy.

\begin{lemma}
	\label{lem:tconj}
	Let $X$ be topologically conjugate to $Y$. Then, $\stab(X) = \stab(Y)$.
\end{lemma}

We can also quickly rule out the realizability of non-normal subgroups.

\begin{lemma}
	\label{lem:no_normal}
	Let $X$ be a non-empty subshift and $H\in\Sub(G)$ a subgroup such that $H\in\stab(X)$. Then, for all $g\in G$, $gHg^{-1}\in\stab(X)$.
	In particular, subgroups that are not normal are not realizable.
\end{lemma}

\begin{proof}
	Let $x\in X$ be a configuration such that $\stab(x) = H$. Then, for any $g\in G$
	$$\stab(g\cdot x) = gHg^{-1}.$$
	Thus, $gHg^{-1}\in\stab(X)$. If $H$ is not normal, there exist $g_0\in G$ and $h\in H$ such that $g_0hg_0^{-1}\notin H$. Suppose $H$ were realizable by $X$. Then, for any $x\in X$, $\stab(g_0\cdot x)\neq H$, which is a contradiction.
\end{proof}

\begin{remark}
	For any group $G$, there is a natural action of $G$ on its space of subgroup by conjugation, that is, $g\cdot H = gHg^{-1}$. The previous lemma shows that $\stab:X\to\Sub(G)$ is a $G$-invariant map. Furthermore, if we take an alphabet $A$ of size at least 2, $\stab(A^G) = \Sub(G)$. Let $a,b\in A$ be two distinct letters. For a subgroup $H\in\Sub(G)$ we can define $x\in X$ by $x(h) = a$ if $h\in H$, and $b$ otherwise. Then, $\stab(x) = H$.
\end{remark}

In contrast to non-normal subgroups, finite index normal subgroups are always realizable.

\begin{lemma}
	Finite index normal subgroups are always realizable.
\end{lemma}

\begin{proof}
    Let $N\trianglelefteq G$ be a finite index normal subgroup, of index $n$. Consider the SFT $X\subseteq \{0,...,n-1\}^{G/N}$, over the finite subgroup $G/N$, comprised of all configurations where there is a unique number for each element of $G/N$. Then, the pull-back $\pi^*(X)\subseteq \{0,...,n-1\}^{G}$ realizes $N$, and by Lemma~\ref{lem:patterns_pull_back} is an SFT.
\end{proof}

Through operations between subshifts we can combine realizable subsets obtain new ones. For this purpose, we define the subgroup-wise intersection of two subsets as

$$\G_1\sqcap\G_2 = \{H_1\cap H_2 \mid H_1\in\G_1, H_2\in\G_2\}.$$

\begin{proposition}
\label{prop:props}
	Let $I$ be a finite set of indices and $(\G_{i})_i$ realizable subsets. Then,
	\begin{enumerate}
		\item $\bigcup_{i\in I}\G_i$ is realizable,
		\item $\bigsqcap_{i\in I}\G_i$ is realizable.
	\end{enumerate}
\end{proposition}

\begin{proof}
	We will prove the case of $|I| = 2$, the general case follows directly. Let $X_i$ be the SFT that realizes $\G_i$ for $i=1,2$.
	\begin{enumerate}
		\item Assuming the alphabets of $X_1$ and $X_2$ are disjoint, take $Y = X_1\cup X_2$. This new subshift is an SFT as the finite union of SFTs is an SFT. Every configuration of both $X_1$ and $X_2$ is contained in $Y$, and therefore $\G_1\cup\G_2\subseteq\stab(Y)$. It is straightforward that $\stab(y)\in\G_1\cup\G_2$ for all $y\in Y$.
		
		\item Define $Y = X_1 \times X_2$. Once again, $Y$ is an SFT as the product of two SFTs in a SFT. Take a subgroup $H = H_1\cap H_2\in\G_1\sqcap\G_2$ and $x_i\in X_i$ such that $\stab(x_i)=H_i$ for $i\in\{1,2\}$. Take $x = (x_1,x_2)\in Y$ and $g\in\stab(x)$. By definition, $g$ must stabilize both $x_1$ and $x_2$. Thus, $g\in\stab(x_1)\cap\stab(x_2)$. Conversely, if an element $g\in G$ stabilizes both $x_1$ and $x_2$, it stabilizes $x$. Therefore, $\stab(x) = \stab(x_1)\cap\stab(x_2)$. An analogous procedure shows that the stabilizer of any $x\in Y$ is the intersection of the stabilizers of its coordinates.
	\end{enumerate}\vspace{-10pt}\end{proof}
Particular instances of this lemma have already appeared in the literature with the purpose of finding strongly aperiodic SFTs. For example in~\cite[Proposition~3.4]{cohen2022strongly} and~\cite[Proposition 2.3]{jeandel2015aperiodic}.

\subsection{Quotients and realizability}

Let $N\trianglelefteq G$ be a non-trivial finitely generated normal subgroup. We want to study how realizable families in the quotient $G/N$ influence realizable families in $G$, and vice-versa.\\

Given a section $\rho:G/N\to G$, a set $S$ of generators for $G/N$, and $T$ a set of generators of $N$, the set $\rho(S)\cup T$ is a finite generating set for $G$. Recall from Section~\ref{sec:canonical} that starting from a subshift $X\subseteq A^{G/N}$ we can define its \define{pull-back} $\pi^*(X)\subseteq A^G$, where $\pi:G\to G/N$ is the quotient map. We will use the pull-back shift to go from realizability in the quotient to realizability in the starting group. Furthermore, when $X$ is a nearest neighbor SFT over $S$ the pull-back is a nearest-neighbor SFT over $\rho(S)\cup T$ (see Lemma~\ref{lem:patterns_pull_back}).\\

We make extensive use of the fact that for all $k_1,k_2,k\in G/N$ we have that $\rho(k_1k_2) = \rho(k_1)\rho(k_2)h$ and $\rho(k^{-1}) = \rho(k)^{-1}h'$ for some $h,h'\in N$.

\begin{proposition}
	\label{prop:quotient}
	Let $N\trianglelefteq G$ be a finitely generated normal subgroup. Let $X$ be a $G/N$-subshift, and $\pi^*(X)$ its pull-back. Then, $\stab(\pi^*(X)) = \rho(\stab(X))N$, where $\rho:G/N\to G$ is any section.
	In particular, if $\G$ is realizable in $G/N$, then $\rho(\G)N = \{\rho(H)N \mid H\in \G\}$ is realizable in $G$. 
\end{proposition}

\begin{proof}
	Fix a section $\rho:G/N\to G$, take $y\in \pi^*(X)$ and define $x_g = y_{\rho(g)}$ for $g\in G/N$. By Lemma \ref{lem:pullpush}, $x\in X$. Take $g\in\stab(y)$, with $k\in G/H$ and $h\in N$ such that $g=\rho(k)h$. For all $k'\in G/N$,
	\begin{align*}
		x(k') = y(\rho(k')) &= (g\cdot y)(\rho(k')) = y(h^{-1}\rho(k)^{-1}\rho(k'))\\
		&= y(h'\rho(k^{-1}k)) \ \text{ for some } h'\in N\\
		&= y(\rho(k^{-1}k')) = x(k^{-1}k')\\
		&= (k\cdot x)(k').
	\end{align*}
	Therefore, $k\in\stab(x)$ and $\stab(y)\subseteq \rho(\stab(x))N$. Conversely, if $g = \rho(k)h\in \rho(\stab(x))N$, for any $k'\in G/N$ and $h'\in N$, 
	\begin{align*}
		(g\cdot y)(\rho(k')h') &= y(h^{-1}\rho(k)^{-1}\rho(k')h')\\
		&= y(h\rho(k^{-1}k)) \ \text{ for some } h\in N\\
		&= y(\rho(k^{-1}k')) = x(k^{-1}k')\\
		&= x(k') = y(\rho(k'))\\
		&= y(\rho(k')h').
	\end{align*}
	Thus, $\stab(y) = \rho(\stab(x))N$. This, in turn, implies that $\stab(\pi^*(X))\subseteq \rho(\G)N$. To see that they are equal, given $x\in X$ we define $y(\rho(k)h) = x(k)$ for all $k\in G/N$ and $h\in N$. Retracing the steps above we can confirm $\rho(\G)N = \stab(\pi^*(X))$.
	
	Finally, if $X\subseteq A^{G/N}$ is a non-empty SFT that realizes $\G\subseteq\Sub(G/N)$; by Lemma \ref{lem:patterns_pull_back} we know $\pi^*(X)$ is a non-empty SFT and realizes $\rho(\G)N$.
\end{proof}

We can also state restrictions in the other direction by making use of the \define{push-forward} subshift (see Section~\ref{subsec:push_forward}).

\begin{proposition}
	\label{prop:vuelta}
	Let $N\trianglelefteq G$ be a finitely generated normal subgroup. Let $X\subseteq \Fix_A(N)$ be a $G$-SFT and $\rho^*(X)$ its push-forward, for any section $\rho:G/N\to G$. Then, $\stab(\rho^*(X)) = \stab(X)/N$. In particular, if $\G$ is realizable in $G$ such that $N\subseteq \bigcap_{H\in\G}H$, then $\G/N = \{K/N \mid K\in \G\}$ is realizable in $G/N$. 
\end{proposition}

\begin{proof}
	For any configuration $y\in\rho^*(X)$, there exists $x\in X$ such that $y = x\circ\rho$. Let us now show that $\stab(y) = \stab(x)/N$. Indeed, given $k\in\stab(y)$, for any $k'\in G/N$ and $h\in N$ we have that
	\begin{align*}
		(\rho(k)\cdot x)(\rho(k')h) &= x(\rho(k)^{-1}\rho(k')h)\\
		&= x(h'\rho(k^{-1}k)) \ \text{ for some } h'\in N\\
		&= x(\rho(k^{-1}k')) = y(k^{-1}k')\\
		&= y(k') = x(\rho(k'))\\
		&= x(\rho(k')h).
	\end{align*}
	In other words, $\rho(k)\in\stab(x)$ and consequently $k\in\stab(x)/N$. Now, let $k\in\stab(x)/N$ (equivalently $\rho(k)\in\stab(x)$). For any $k'\in G/N$ we have
	\begin{align*}
		(k\cdot y)(k') &= y(k^{-1}k')\\
		&= x(\rho(k^{-1}k)) = x(h'\rho(k)^{-1}\rho(k))\ \text{ for some } h'\in N\\
		&= x(\rho(k)^{-1})\rho(k')) = x(\rho(k'))\\
		&= y(k').
	\end{align*}
	
	Therefore, $\stab(y) = \stab(x)/N$. Finally, if we take $x\in X$ we can define $y\in\rho^*(X)$ by $y(k) = x(\rho(k))$, by retracing the previous steps we obtain that $\stab(\rho^*(X))=\G/N$. 
	
	Finally, if $X\subseteq A^G$  is a non-empty SFT that realizes $\G$; because $N\subseteq H$ for all $H\in\G$, $X\subseteq\Fix_A(N)$. Furthermore, by Lemma \ref{lem:tconj} we can take $X$ to be a nearest neighbor SFT with respect to the generating set $\rho(S)\cup T$. By Lemma ~\ref{lem:patterns_push_forward}, $\rho^*(X)$ is a non-empty SFT that realizes $\G/N$.
	
\end{proof}

Combining both propositions we obtain a characterization of realizable finitely generated normal subgroups.

\begin{theorem}
	\label{thm:quotient_ap}
	Let $N\trianglelefteq G$ be a non-trivial finitely generated normal subgroup. Then, $N$ is realizable in $G$ if and only if $G/N$ admits a strongly aperiodic SFT.
\end{theorem}

\begin{proof}
	If $N$ is realizable by a $G$-SFT $X$, by Proposition \ref{prop:vuelta}, its push-forward shift $\rho^*(X)$ realizes $\{1_{G/N}\}$. Conversely, if $\{1_{G/N}\}$ is realized by a $G/N$-SFT $Y$, then by Proposition~\ref{prop:quotient} its pull-back $\pi^*(Y)$ realizes $N$.
\end{proof}

As a consequence, we find many examples of non-realizable subgroups.

\begin{corollary}
	Let $G$ be a finitely generated group, and a finitely generated normal subgroup $N\trianglelefteq G$. If $G/N$ is virtually free, then $N$ is not realizable in $G$. In particular, every torsion-free nilpotent group has normal subgroups that are not realizable.
\end{corollary}

A particular class where this occurs is in the class of \define{indicable} groups. A group $G$ is said to be indicable if it admits an epimorphism $G\twoheadrightarrow\Z$. By the previous corollary, if $G$ is indicable and the kernel of the epimorphism is finitely generated, the kernel is not realizable. For example, finitely generated torsion-free nilpotent groups are indicable~\cite{higman1940units}, and all of their subgroups are finitely generated. Similarly, if an indicable group does not contain the free semi-group on two generators, the kernel of the epimorphism will be finitely generated (see~\cite[Lemma 3]{benli2012indicable}). On the other hand, in the case of just infinite groups all non-trivial normal subgroups are realizable, as they all have finite index.


\subsection{No restrictions}

As we have seen, being an SFT imposes heavy restrictions on realizability. But what happens if we just ask for a subshift? By combining our previous results with the existence of strongly aperiodic subshifts on every countable group, we can answer the question.

\begin{proposition}
\label{prop:no_reservations}
	Let $G$ be a finitely generated group and take a subgroup $H\leq G$. There exists a $G$-subshift that realizes $H$ if and only if $H$ is normal.
\end{proposition}

\begin{proof}
	By Lemma \ref{lem:no_normal}, if $H$ is not normal it is not realizable. Suppose $H$ is a normal subgroup. By \cite[Theorem 2.4]{aubrun2019realization}, we know every countable group admits a strongly aperiodic subshift. In particular, there exists a $G/H$-subshift $X_{G/H}$ that realizes $\{1_{G/H}\}$. By Proposition~\ref{prop:quotient} the pull-back shift, $\pi^*(X)$, realizes $H$.
\end{proof}

\section{Computational restrictions}
\label{sec:computational}

Recall from Theorem~\ref{thm:jeandel} that if a finitely generated recursively presented group $G$ admits a strongly aperiodic SFT, $\WP(G)$ must be decidable. We show that a similar result can be obtained for the realizability of subgroups.\\

Let $G$ be a finitely generated group of rank $n$, and $\pi:\F_n \to G$ the canonical epimorphism. For a $G$-subshift $X\subseteq A^G$, let $\pi^*(X)\subseteq A^{\F_n}$ be its pull-back, where for all $y\in\pi^*(X)$ there exists $x\in X$ such that $y = x\circ\pi$. As $\ker(\pi)$ is not necessarily finitely generated, $\pi^*(X)$ may not be an SFT. Nevertheless, it is an effective subshift when $G$ is recursively presented.

\begin{lemma}[\cite{jeandel2015aperiodic} Prop. 1.3 and Prop. 1.7]
	\label{lem:vacio_eff}
	Let $G$ be a finitely generated recursively presented group. Given an effective set of forbidden patterns $\mathcal{F}$ through an enumeration, there is a semi-algorithm that halts if and only if $\X_{\mathcal{F}} = \varnothing$.
\end{lemma}

We link the realizability of a subgroup to its subgroup membership problem. 

\begin{definition}
	Let $G$ be a finitely generated group and $S$ a generating set.
	The \define{subgroup membership problem of $H$} in $G$ asks, given a set of words $u, w_i\in S^*$ for $i\in\{1,...,k\}$ such that $H=\langle \overline{w}_1, ..., \overline{w}_k\rangle$, whether $\overline{u}\in H$.
\end{definition}

\begin{lemma}
	\label{lem:memb}
	Let $H$ be a finitely generated group of a recursively presented group $G$. Then, there is a semi-algorithm for the subgroup membership problem of $H$ in $G$.
\end{lemma}

\begin{proof}
	Because $G$ is recursively presented we know $\WP(G)$ can be enumerated (Proposition~\ref{prop:WP}). Now given an input $u, w_i\in S^*$ for the subgroup membership problem of $H$, we know $\overline{u}\in H$ if and only if there exists a word $w\in\{w_1^{\pm 1}, ..., w_n^{\pm1}\}^*$ such that $uw^{-1}=_G \varepsilon$. The semi-algorithm consists in enumerating all such words $w$ and seeing if $uw^{-1}$ appears in the enumeration of $\WP(G)$.
\end{proof}

\begin{theorem}
	\label{thm:jean_gen}
	Let $G$ be a finitely generated recursively presented group and $H$ a finitely generated subgroup. If $H$ is realizable, then the subgroup membership problem of $H$ in $G$ is decidable.
\end{theorem}

\begin{proof}
	Let $X$ be a $G$-SFT that realizes $H$ and $Y = \pi^*(X)$ its pull-back to $\F_n$, where $n$ is the rank of $G$. From Proposition~\ref{prop:quotient} we know that for every $y\in Y$, $\stab(y) = \pi^{-1}(H)$.  Now, let $u, w_i\in S^*$ be an input to the subgroup membership problem for $H$ in $G$. By reducing $u$ we can suppose that $u\in\F_n$. We define the $\F_n$-SFT,
	$$Z = \{x\in A^{\F_n}\mid \forall g\in\F_n,\ u\cdot x(g) = x(g)\}.$$
	 This way, $Y\cap Z = \varnothing$ if and only if $u\notin \pi^{-1}(H)$, i.e. $u$ does not belong to $H$ in $G$. Because $Y$ is effective and $Z$ is an SFT, by Lemma \ref{lem:vacio_eff} there is a semi-algorithm to determine if $Y\cap Z$ is empty. Thus, there is a semi-algorithm to determine if an element does not belong to a group. Paired with Lemma \ref{lem:memb}, this implies the subgroup membership problem of $H$ in $G$ is decidable.
\end{proof}

\begin{example}
	\label{ex:Rips}
	Using Rip's construction~\cite{rips1982subgroups} with a finitely presented group with undecidable word problem, it is possible to obtain a hyperbolic group with a finitely generated normal subgroup with undecidable subgroup membership problem. This argument is usually attributed to Sela \cite{gromov1992asymptotic}.
\end{example}

\begin{example}
	\label{ex:Mihailova}
	It is possible to find subgroup with undecidable membership problem within $\F_n\times\F_n$, due to an argument by Mihailova~\cite{mihailova1968occurrence}. Given a finitely generated group $G$ of rank $n$, we define its Mihailova subgroup as
	$$M(G) = \{(w_1,w_2)\in\F_n\times\F_n \mid w_1 =_{G} w_2\}.$$
	Notice that if $G$ is finitely presented by a set of generators $S$ and relations $R$, the set of generators including $\{(s,s)\}_{s\in S}$ and $\{(1,r)\}_{r\in R}$ are a generating set for $M(G)$. Then, the subgroup membership problem for $M(G)$ in $\F_{n}\times\F_{n}$ is decidable if and only if $G$ has decidable word problem.
\end{example}

\subsection{Restrictions on $\Z^d$}
\label{sec:restrictionsZd}
There are other types of computational restrictions on realizability that do not involve membership problems. A particular family of these restrictions comes from the study of periodicity on $\Z^d$-SFTs.
\subsubsection{One-dimensional subshifts}

Nearest neighbor $\Z$-SFTs have a rigid structure. Given a set of nearest neighbor patterns $\Fo$ over an alphabet $A$, we define its corresponding \define{tileset graph}, $\Gamma_{\Fo}$, by the set of vertices $A$ and edges given by $(a,b)\in A^2$ such that $(a,b)\not\in\Fo$, where $a$ is its initial vertex, $b$ its final vertex and $s$ its label. This way, $\X_\Fo\subseteq A^\Z$ is the set of bi-infinite directed walks on $\Gamma_{\Fo}$.

Conversely, given a finite directed graph $\Gamma = (V,E)$, we define its associated nearest neighbor SFT $X_{\Gamma}\subseteq V^{\Z}$ as
$$X_{\Gamma} = \{x\in V^{\Z} \mid (x(k), x(k+1))\in E\}.$$
See~\cite{lind2021introduction} for a classic introduction to $\Z$-subshifts and the consequences of the correspondence between nearest neighbor SFTs and directed graphs.

For simplicity, instead of looking at stabilizers, we will look at the corresponding multiples of the subgroups. For a subshift $X\subseteq A^\Z$ we define its \define{multiples} by
$$M(X) = \{p\in\N \mid \exists x\in X, \stab(x) = p\Z\}$$

The set of periods of nearest neighbor $\Z$-SFT have been completely classified. We start by looking at the particular case where the $\Gamma$ is strongly connected.
\begin{theorem}[\cite{ali2018periods}]
\label{thm:periodic}
For $P\subseteq \N\setminus\{0\}$ the following are equivalent:
\begin{itemize}
    \item There exists a strongly connected graph $\Gamma$ such that $P = M(X_{\Gamma})\setminus\{0\}$,
    \item $P$ is a singleton or there exists $k\in\N$ and finite sets $F, F'\subseteq\N$ such that $P = F \cup (k\N\setminus F')$.  
\end{itemize}
\end{theorem}

To understand this result let us see how to realize the subsets mentioned on the previous theorem. First off, singletons. For $P = \{p\}$, with $p\neq 0$, the graph $\Gamma_p$ consisting of a directed simple cycle of length $p$ defines an SFT that realizes $P$. Explicitly,
	$$X_{\Gamma_p} = \{x\in\{0, ..., p-1\}^{\Z} \mid x_{i+1} = x_{i}+1\!\mod p\}.$$
Using these subshifts and Proposition~\ref{prop:props}, we can realize any finite subset of $\N$ not containing $0$. Next, for subsets of the form $P = \{n+k\mid n\in\N\}$, define the graph $\Gamma_{n+k}$ as a directed cycle with $k$ vertices, with additional vertices and edges such that every edge of the cycle is a triangle oriented in the same direction as the cycle. For example, for $k=4$ we obtain the graph from Figure~\ref{fig:grafo}.

\begin{figure}[H]
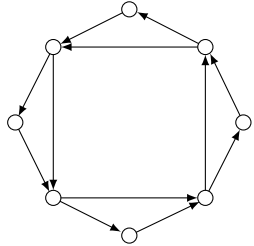

		\centering
		\includestandalone[scale=1]{figures/grafo1}
		\caption{The graph $\Gamma_{n+4}$ that defines the SFT with set of multiples $\{n+4\mid n\in\N\}$.}
		\label{fig:grafo}
	\end{figure}

A simple computation shows that $M(X_{\Gamma_{n+k}}) = \{n+k\mid n\in\N\}$. Finally, take $a\geq 1$. If we replace every edge in $\Gamma_{n+k}$ by a directed path of length $a$ in the same direction as the edge, we obtain a new graph $\Gamma_{a(n+k)}$. As this new graph's subscript suggests, its corresponding SFT realizes the subset $\{a(n+k) \mid n\in\N\}$.\\

To complete the characterization of multiples (and therefore stabilizers), we must understand the role of $0$. 
\begin{lemma}
\label{lem:zero}
Let $X_{\Gamma}\subseteq A^{\Z}$ be a nearest neighbor SFT. Then, $\{0\}\in\stab(X_{\Gamma})$ if and only if $\stab(X_{\Gamma})$ is infinite.
\end{lemma}
\begin{proof}
    If $\{0\}\in\stab(X_{\Gamma})$, there exists an aperiodic configuration $x\in X_{\Gamma}$. Then, there exists $a\in A$ that appears infinitely often in $x$. Let $W$ be the set of words that occur in $x$ between two consecutive occurrences of $a$. Because $x$ is aperiodic, there exist two words $w_1, w_2\in W$. By their definition, both words define cycles in $\Gamma$ based at the vertex $a$. Consider $i, j\geq 1$ and the configuration $y_{i,j} = (w_1^iw_2^j)^{\infty}$. Because $w_1$ and $w_2$ are cycles in $\Gamma$, $y\in X_{\Gamma}$. Furthermore, $y$ is a periodic point of minimal period $i|w_1| + j|w_2|$. Therefore, $\stab(X_{\Gamma})$ is infinite.

    Conversely, suppose $\stab(X_{\Gamma})$ is infinite. Because every periodic point corresponds to a cycle in $\Gamma$, there exist two cycles $\gamma_1$ and $\gamma_2$ based at the same vertex. Let $\mu\in\{1,2\}^{\Z}$ be a bi-infinite aperiodic word (such as the Thue-Morse word, or any Sturmian word). We define
	$$x = ... \gamma_{\mu_{-2}}\gamma_{\mu_{-1}}.\gamma_{\mu_{0}}\gamma_{\mu_{1}}\gamma_{\mu_{2}} ...$$
	By definition $x\in X_{\Gamma}$, and as $\mu$ is aperiodic, $\stab(x) = \{0\}$.
\end{proof}

The last ingredient is what happens when $\Gamma$ is not strongly connected. In this case, the set of stabilizers is the union of the stabilizers associated to each strongly connected component of $\Gamma$ (for more details see~\cite{lind2021introduction}). By joining this fact with Theorem~\ref{thm:periodic} we obtain:
\begin{corollary}[Corollary 1~\cite{ali2018periods}]
\label{cor:periods}
    For $P\subseteq \N\setminus\{0\}$ the following are equivalent:
\begin{itemize}
    \item There exists a graph $\Gamma$ such that $P = M(X_{\Gamma})\setminus\{0\}$,
    \item $P$ there exist $m\in\N$, $k_i\in\N$ and finite sets $F, F_i\subseteq\N$ such that $P = F \cup \bigcup_{i=1}^m (k_i\N\setminus F_i)$.  
\end{itemize}
\end{corollary}

\vspace{0.2cm}
We can now bring everything together to obtain a characterization of realizable sets.
\begin{theorem}
    $P\subseteq\N$ is realizable if and only if either
    \begin{itemize}
        \item $P$ is finite and $0\notin P$,
        \item there exists $m\in\N$, $(k_i)_{i=1}^m\in\N^m$, $(a_i)_{i=1}^m$, and a finite set $F\subseteq\N$ with $0\in F$ such that 
        $$P = F\cup \bigcup_{i=1}^m\{a_i(n + k_i) \mid n\in\N\}.$$
    \end{itemize}
\end{theorem}
\begin{proof}
    Recall that for $p\neq 0$, $k\in\N$ and $a\geq 1$ we have that $M(X_{\Gamma_p}) = \{p\}$ and $M(X_{\Gamma_{a(n+k)}})=\{a(n + k) \mid n\in\N\}$. By Proposition~\ref{prop:props}, finite sets without $0$ and sets of the form $F\cup \bigcup_{i=1}^m\{a_i(n + k_i) \mid n\in\N\}$ as in the statement are realizable. 

    Now, consider a realizable set $P\subseteq\N$. If $P$ is finite, by Lemma~\ref{lem:zero} $0\notin P$. If $P$ is infinite, by Corollary~\ref{cor:periods} there exist $m\in\N$, $k_i\in\N$ and finite sets $F, F_i\Subset \N$ such that
    $P = F \cup \bigcup_{i=1}^m (k_i\N\setminus F_i)$. Notice that for every $i\in\{1,...,m\}$ these exists $F'_i\Subset\N$, $k\in\N$ and $a_i\geq 1$ such that $$k_i\N\setminus F_i = F'_i \cup \{a_i(n+k_i) \mid n\in\N\}.$$
    Then, by defining $F' = F\cup\bigcup_{i=1}^m F'_i$ we obtain 
    $$P = F' \cup \bigcup_{i=1}^m\{a_i(n + k_i) \mid n\in\N\}.$$
    where $0\in F'$ by Lemma~\ref{lem:zero}.
\end{proof}

This characterization gives us a computational restriction. Given a set $F\subseteq \N$ we define its corresponding unary language as $un(F) = \{a^n \mid n\in F\}$. Parihk's Theorem for unary languages states that a unary language $L = \{a^n \mid n\in F\}$ is regular if and only $F$ is a \define{semi-linear} set. By the previous theorem, all realizable sets are semi-linear. 

\begin{corollary}
Let $X$ be a $\Z$-SFT. Then, $un(M(X))$ is regular.
\end{corollary}

\subsubsection{Multi-dimensional subshifts}
 
We say two element $u,v\in\Z^2$ are equivalent, which we denote by $v\sim u$, if there exists $\lambda\neq 0$ such that $v = \lambda u$. We call equivalence classes under $\sim$, \define{slopes}, and denote them by $[v]$. We denote the set of all slopes in $\Z^d$ by $S(\Z^d)$. Given a $\Z^d$-SFT $X$, we define its set of slopes as 
$$Sl(X) = \{[v]\in S(\Z^d) \mid \exists x\in X,\ \stab(x) = v\Z\}.$$

Jeandel, Moutot and Vanier showed that the set of slopes of $\Z^2$-SFTs are exactly $\Sigma^0_1$ subsets of $S(\Z^2)$, and that the set of slopes of $\Z^3$-SFTs are exactly $\Sigma^0_2$ subsets of $S(\Z^3)$~\cite{jeandel2020slopes}. There are further restrictions in the case of $\Z^2$ if we encode the stabilizers differently. For $X$ a $\Z^2$-SFT, we define,
\begin{itemize}
	\item the set of \define{full-periods} of $X$ as $\mathfrak{P}(X) = \{n\in\Z \mid \exists x\in X,\ \stab(x)= (n\Z)^2\}$,
	\item the set of \define{1-periods} of $X$ as $\mathfrak{P}_1(X) = \{v\in\Z\times\N\setminus\{(0,0)\} \mid \exists x\in X,\ \stab(x)= v\Z\}$,
	\item the set of \define{horizontal periods} of $X$ as $\mathfrak{P}_h(X) = \{n\in\Z \mid \exists x\in X,\ \stab(x)= k\Z\times\{0\}\}$.
\end{itemize}

As we did for subsets $F\subseteq \N$, we define the language associated to $F'\subseteq \Z\times\N$ as $un(F') = \{a^pb^q\mid (p,q)\in F'\}\cup\{a^pc^q \mid (-p,q)\in F'\}$. Jeandel and Vanier showed that a set $F\subseteq\N$ is the set of full-periods of an SFT if and only if $un(F)\in\textbf{NP}$, that a set $F'\subseteq \Z\times\N$ is the set of 1-periods of an SFT if and only if $un(F')\in\textbf{NSPACE}(n)$, and that $F''\subseteq\N$ is the set of horizontal periods of an SFT if and only if $un(F'')\in\textbf{NSPACE}(n)$~\cite{jeandel2015charact}.\\

Do these restrictions provide a full description of realizable subsets $\G\subseteq\Sub(\Z^2)$? The answer is no. It suffices to take the singleton $\G = \{(p,0)\Z\}$ for any non-trivial $p\in\Z$, which satisfies all the previous conditions but is not realizable by Lemma~\ref{lem:OGPR}. In the next section we will see that this can be taken further: if $\G$ consists exclusively of one dimensional subspaces it is not realizable.

\section{Periodic rigidity}
\label{sec:PeriodicRigidity}

There is a strange phenomenon with regards to aperiodicity in $\Z^2$. As we saw in Lemma~\ref{lem:OGPR}, if there is a configuration with a horizontal period on a $\Z^2$-SFT, there must be a periodic configuration within the SFT. This is part of a larger phenomenon, where no family of subgroups where every subgroup has infinite index is realizable.

\begin{lemma}
	\label{lem:PR_Z2}
	Every weakly aperiodic $\Z^2$-SFT is strongly aperiodic.
\end{lemma}

\begin{proof}
	Let $X\subseteq A^{\Z^2}$ be a weakly aperiodic nearest neighbor SFT on $\Z^2$. If $X$ is not strongly aperiodic, there exists a configuration $x\in X$ stabilized by a non-trivial element $v =(p,q)\in\Z^2$. Suppose without loss of generality that $q > 0$ and consider the portion of the plane $P$ given by the strip $\Z\times\{0,...,q-1\}$. Because $x$ is stabilized by $v$ we have $x|_P = x|_{v + P}$. Now, if we cut $P$ into blocks of width $|p|$, and look at their tiling $B_i = x|_{\{i, ..., i+|p|-1\}\times\{0,...,q-1\}}$, there must exist $i_1$ and $i_2$ such that $B_{i_1} = B_{i_2}$ as the alphabet is finite. Define $B = x|_{\{i_1, ..., i_2-1\}\times\{0,...,q-1\}}$ and the configuration $y\in A^{\Z^2}$ which contains the bi-infinite repetition of $B$ on the strip $P$, and is completed by stacking strips with the appropriate shift by a multiple of $p$ (see Figure~\ref{fig:PR_Z2}). Because $X$ is a nearest neighbor SFT, $y$ belong to $X$ and is stabilized by the subgroup $(i_2 - i_1)\Z\times p\Z$. This contradicts the fact that $X$ is weakly aperiodic.
\end{proof}
	\begin{figure}[H]
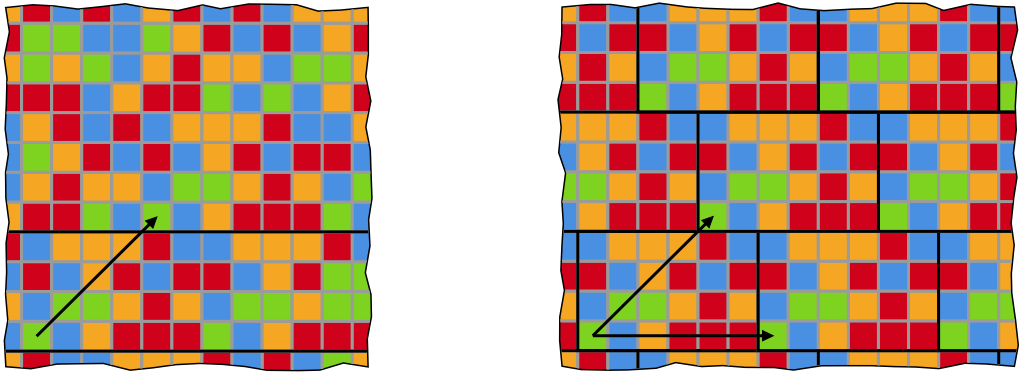

		\centering
		\definecolor{blue}{rgb}{0,0,1}
		\definecolor{red}{rgb}{1,0,0}
		\begin{subfigure}{0.45\textwidth}
			\includestandalone[scale=1.5]{figures/PR_Z2_1}
		\end{subfigure}
		\begin{subfigure}{0.45\textwidth}
			\includestandalone[scale=1.5]{figures/PR_Z2_2}
		\end{subfigure}
		\caption{For SFTs in $\Z^2$, having a configuration with non-trivial stabilizer (on the left) implies the existence of a periodic configuration (on the right). This is done by finding a repeating motif on the strip defined by the period vector, and repeating this motif in a way compatible with the forbidden patterns of the nearest neighbor SFT.}
		\label{fig:PR_Z2}
	\end{figure}
For which other groups does this hold? Although there have been examples of groups which have weakly aperiodic SFTs that are not strongly aperiodic in the past, for $\Z^d$ with $d\geq 3$ for example~\cite{culik1995aperiodic}, the first explicit construction is due to Moutot and Esnay for Baumslag-Solitar groups~\cite{esnay2022weakly}. In this section we study necessary and sufficient conditions for groups to exhibit this behavior.

\begin{definition}
	We say a group $G$ is \define{periodically rigid} if every weakly aperiodic $G$-SFT is strongly aperiodic.
\end{definition}

This is equivalent to saying that the only non-empty SFT $X$ such that all stabilizers have infinite index, are those such that $\stab(X)=\{1\}$. In particular, if $G$ is periodically rigid, then no infinite index non-trivial subgroup is realizable. By Lemma~\ref{lem:PR_Z2}, $\Z^2$ is periodically rigid. In addition, by vacuity, all virtually $\Z$ groups are periodically rigid.

Although not with our terminology, Pytheas-Fogg posed the following question.

\begin{question}[\cite{pytheas2022conjecture}]
	Is a finitely generated group periodically rigid if and only if it is either virtually $\Z$ or virtually~$\Z^2$?
\end{question}

It has already been shown that some classes of groups admit weakly but not strongly aperiodic SFTs. These are the following:
\begin{itemize}
	\item $\Z^d$ for $d > 2$,
	\item Baumslag-Solitar groups $BS(m,n)$ with $|n|,|m|\neq 1$~\cite{esnay2022weakly},
	\item Free groups, \cite{piantodosi2008free},
	\item Hyperbolic groups~\cite{gromov1987hyperbolic,coornaert2006symbolic},
	\item Groups with two or more ends~\cite{cohen2017large},
	\item the Lamplighter group~\cite{cohen2020lamplighters},
\end{itemize}

Let us establish some inheritance properties of periodically rigid groups.

\begin{proposition}
	\label{prop:virt_PR}
	Take $G$ a torsion-free group, and $H\leq G$ a finite index subgroup. If $H$ is periodically rigid, then $G$ is periodically rigid.
\end{proposition}

\begin{proof}
	Suppose there exists $X\subseteq A^G$ a weakly but not strongly aperiodic $G$-SFT. For a set of right coset representatives $R$, take the $R$-higher power shift $X^{[R]}$. By Lemma \ref{lem:patterns_higher_block}, $X^{[R]}$ is an $H$-SFT. Furthermore, if we take $y\in X^{[R]}$ and its corresponding configuration $x\in X$, we have that $\stab(y)\subseteq\stab(x)$.
	
	Now, because $X$ is not strongly aperiodic, there exists $g\in G\setminus\{1_G\}$ and $x\in X$ such that $g\in\stab(x)$. As $H$ is of finite index, and $G$ is torsion-free there exists $n \geq 1$ such that $g^n \in H\setminus\{1_G\}$. Define $y\in X^{[R]}$ by $y(h)(r) = x(hr)$ for all $h\in H$ and $r\in R$. Then,
	\begin{align*}
		(g^n\cdot y)(h)(r) &= y(g^{-n}h)(r) = x(g^{-n}hr)\\
		&= x(hr) = y(h)(r),
	\end{align*}
	and thus $g^{n}\in\stab(y)$. Because $H$ is periodically rigid, this means that there exists $z\in X^{[R]}$ such that $\stab(z)$ has finite index in $H$. If we denote $x\in X$ the configuration such that $z(h)(r) = x(hr)$, $\stab(x)$ contains a finite index subgroup and is therefore a finite index subgroup itself. Thus, $x$ is a periodic configuration of $X$. This is a contradiction, as $X$ was supposed to be weakly periodic.
\end{proof}

\begin{example}
	The fundamental group of the Klein bottle, which is given by 
	$$\pi_1(K) = BS(1,-1) = \langle a,b \mid abab^{-1}\rangle,$$
	is torsion-free virtually $\Z^2$ and therefore periodically rigid by the previous proposition. 
\end{example}

This last example shows Pytheas-Fogg's question is incomplete, and allows us to state the following conjecture.
\begin{conjecture}[Conjecture~\ref{intro:conjetura}]
	\label{conj:periodic_rigid}
	A finitely generated group is periodically rigid if and only if it is either virtually $\Z$ or torsion-free virtually $\Z^2$.
\end{conjecture}

\begin{remark}
	Notice that the previous conjeture implies Conjecture~\ref{conj:CarrollPenland} concerning weakly aperiodic SFTs. Indeed, if there existed a non-virtually $\Z$ group that does not admit a weakly aperiodic SFT, it would be periodically rigid.
\end{remark}

What can we say about the periodic rigidity of a group, from the periodic rigidity of its subgroups or quotients? To answer this question, we use of a result by Barbieri, that links stabilizing elements in the free-extension of a shift to the stabilizing elements of the shift. Let us introduce some notation. Given an element $g\in G$, we define its \define{conjugacy class} as
$$\Cl(g) = \{tgt^{-1} \mid t\in G\}.$$
Next, we define the set of \define{roots} of a subgroup $K\leq G$ as
$$R_G(K) = \{g\in G \mid \exists n\in\N, g^n\in K\}.$$
Finally, given a $G$-subshift $X$, we define the set of \define{free elements} of the group action as
$$\Free(X) = G\setminus \bigcup_{x\in X}\stab(x) = \{g\in G \mid g\cdot x \neq x, \ \forall x\in X\}.$$

With all these elements at hand, we state Barbieri's result that characterizes how the stabilizers of the free extension of a subshift behave.  It also holds for non-finitely generated groups.
\begin{theorem}[\cite{barbieri2023aperiodic_non_fg}]
	\label{thm:freeX}
	Take a group $G$, a subgroup $H\leq G$, and an $H$-subshift $X$. Then, $g\in\Free(\upa{X})$ if and only if $\Cl(g)\cap R_{G}(\Free(X))\neq \varnothing$.
\end{theorem}

As a consequence of Lemma~\ref{lem:free_extension}, the free extension of a weakly aperiodic SFT is weakly aperiodic. We will use the previous theorem to determine when the free extension is not strongly aperiodic, to find properties of periodically rigid groups.

\begin{proposition}
	\label{prop:PR_torsion-free}
	Let $G$ be a finitely generated group with a torsion-free subgroup $H\leq G$ that admits a weakly aperiodic SFT, and $g\in G\setminus H$ with torsion. Then, $G$ is not periodically rigid.
\end{proposition}

\begin{proof}
	As $H$ is torsion-free we have that for all $m\in \N$, $g^{m}\notin H\setminus\{1\}$, as every power of an element with torsion has torsion. Furthermore, for all $t\in G$ and $m\in\N$, $tg^mt^{-1}\notin H\setminus\{1\}$, as every conjugate of an element with torsion has torsion. Because  $\Free(X)\subseteq H\setminus\{1\}$, we arrive at $\Cl(g)\cap R_{G}(\Free(X)) = \varnothing$. Then, by Theorem \ref{thm:freeX}, $g\notin\Free(\upa{X})$. As $X$ being weakly aperiodic implies $\upa{X}$ is weakly aperiodic, but $\Free(\upa{X})\neq G\setminus\{1\}$, $G$ is not periodically rigid.
\end{proof}

\begin{proposition}
	Let $G_1$ be an infinite finitely generated group that admits a weakly aperiodic SFT. If $G_2$ is another finitely generated group, then $G_1\rtimes G_2$ is not periodically rigid.
\end{proposition}

\begin{proof}
	Let us denote $G = G_1\rtimes G_2$ and $H_1$ and $H_2$ the subgroups of $G$ such that $H_i \simeq G_i$. Let $X$ be a weakly aperiodic $H_1$-SFT and $Y = \upa{X}$ its free extension to $G$. $Y$ is a weakly aperiodic $G$-SFT. As $G$ is a semi-direct product, $H_1\cap H_2 = \{1\}$. By taking $g\in H_2\setminus\{1\}$, we know that for any $t\in G$ and $n \geq 1$ we have $tg^nt^{-1}\notin H_1\setminus\{1\}$, as $H_1$ is normal. This means $\Cl(g)\cap R_G(\Free(X))$ is empty because $\Free(X)\subseteq H_1\setminus\{1\}$. By Theorem \ref{thm:freeX}, there exists $y\in Y$ such that $g\in\stab(y)$. Thus, $Y$ is not strongly aperiodic.
\end{proof}

\begin{lemma}
	\label{lem:PR_quotient}
	Let $N\trianglelefteq G$ be a non-trivial finitely generated normal subgroup. Then, if $G/N$ admits a weakly aperiodic SFT, $G$ is not periodically rigid.
\end{lemma}

\begin{proof}   
	Let $X\subseteq A^{G/N}$ be a weakly aperiodic SFT, and let $\pi^*(X)\subseteq A^G$ be its pull-back. From Proposition~\ref{prop:quotient}, we know that $\stab(\pi^*(X)) = \rho(\stab(X))N$, for any section $\rho:G/N\to G$.
	
	Suppose there is $y\in\pi^*(X)$ such that $\stab(y)$ has finite index. Then, $x\in X$ defined as $x(k) = y(\rho(k))$ for every $k\in G/N$, would have stabilizer $\stab(y)/N$ of finite index, which is a contradiction. Finally, $\stab(y)$ is non-trivial as it contains $N$.
\end{proof}

\begin{lemma}
	\label{lem:PR_exact}
	Let $G$ be a group that admits an exact sequence given by 
	$$1 \to N \to G\to H\to 1,$$
	where $N$ admits a weakly aperiodic SFT and $H$ has a torsion-free element. Then, $G$ is not periodically rigid.
\end{lemma}

\begin{proof}
	Let $X$ be a weakly aperiodic $N$-SFT and $g\in G$ an element that maps to a free generator of the quotient $G/N \simeq H$. Then, $g^k\notin N$ for all $k \neq 0$, and furthermore $tg^kt^{-1}\notin N$ for all $t\in G$, as $N$ is normal. This fact can be translated to the expression $\Cl(g)\cap R_{G}(N\setminus\{1\}) = \varnothing$, which by Theorem \ref{thm:freeX} means there exists $y\in \upa{X}$ such that $g\in\stab(y)$. Therefore, $\upa{X}$ is a weakly but not strongly aperiodic $G$-SFT.
\end{proof}

\section{Periodic rigidity of virtually nilpotent and polycylic groups}
\label{subsec:nilp_poly}

The objective of this section is to prove Conjecture~\ref{conj:periodic_rigid} holds for the both the class of virtually nilpotent groups and the class of polycyclic groups. To do this, we will first look at the properties these groups satisfy.

\subsection{Definitions, properties and Hirsch length}
\label{subsec:nilpotent}

Let $G$ be a group. For each $i\in\N$ inductively define $Z_i(G)$ as
$$Z_{i+1}(G) = \{g\in G \mid [g,h]\in Z_i(G), \forall h\in G\},$$
where $Z_0(G) = \{1_G\}$. The set $Z(G) = Z_1(G)$ is called the \define{center} of $G$ and, by definition, is the set of elements that commute with every element in $G$. We say a group is \define{nilpotent} if there exists $n\geq 0$ such that $Z_{n}(G) = G$. In this case we also say $G$ has an \define{upper central series} defined by the sequence of normal subgroups,
$$\{1_G\}\trianglelefteq Z(G) \trianglelefteq Z_2(G)\trianglelefteq\ ...\ \trianglelefteq Z_n(G) = G,$$
where $Z_{i+1}/Z_{i} = Z(G/Z_i)$.\\

A similarly defined family of groups is the family of polycyclic groups. A group $G$ is \define{polycyclic} if it admits a series
$$\{1_G\}= G_0\trianglelefteq G_1 \trianglelefteq\ ...\ \trianglelefteq G_n = G,$$
for some $n\geq 1$, such that the quotient $G_{i+1}/G_{i}$ is a cyclic group (finite or infinite).

\begin{example}
	All nilpotent groups are polycyclic, but the converse is not true. The group $\Z^2\rtimes_{M}\Z$ with 
	$$M=\begin{bmatrix}2 & 1 \\ 1 & 1\end{bmatrix}$$ is polycyclic but not nilpotent. Its center is trivial as the matrix has no non-trivial fixed points.
\end{example}

Other examples of polycyclic groups can be obtained from the Auslander-Swan Theorem which states that a polycyclic groups are exactly solvable subgroup of $GL(n,\Z)$. The necessary condition of this theorem was proven by Mal'cev~\cite{malcev1956solvable}, and the sufficient one by Auslander and Swan~\cite{auslander1967problem,swan1967representations}. A proof can be found in~\cite{segal1983polycyclic}.\\

Polycyclic satisfy properties similar to the one satisfied by nilpotent groups. Subgroups and quotients of polycyclic groups are polycyclic, and every polycylic group contains a finite index torsion-free polycylic subgroup. Furthermore, subgroups of polycyclic groups are always finitely generated. This last property is crucial when we want to define SFTs from SFTs on quotients (see Lemma~\ref{lem:patterns_pull_back}).

The key tool when working with polycyclic groups is their \define{Hirsch length}. For a polycyclic group $G$, this length, denoted $h(G)$, is equal to the number of infinite cyclic quotients in its groups series. We make proofs by induction over the Hirsch length by using the following properties.

\begin{proposition}
	\label{prop:hirsch}
	Let $G$ be a polycyclic group. The following hold,
	\begin{itemize}
		\item for a subgroup $H\leq G$, $h(H)\leq h(G)$,
		\item for a normal subgroup $N\trianglelefteq G$, $h(G) = h(N) + h(G/N)$,
		\item $h(G) = 0$ if and only if $G$ is finite,
		\item $h(G) = 1$ if and only if $G$ is virtually $\Z$,
		\item $h(G) = 2$ if and only if $G$ is virtually $\Z^2$,
		\item $h(\Z^d) = d$.
	\end{itemize}
\end{proposition}

For a proof of this proposition and further properties of polycyclic groups see~\cite{segal1983polycyclic}. 

\subsection{Aperiodic SFTs for nilpotent and polycyclic groups}

In this section we prove that all polycyclic groups and all virtually nilpotent groups verify Conjecture~\ref{conj:periodic_rigid}. To do this, we make an induction over the Hirsch length of a group, as was done in~\cite{jeandel2015aperiodic} for strongly aperiodic SFTs.

\begin{theorem}
	\label{thm:PR_nil}
	Finitely generated infinite nilpotent groups are periodically rigid if and only if they are not virtually $\Z$, or $\Z^2$.
\end{theorem}

\begin{proof}
	Let $G$ be a nilpotent group that is neither virtually $\Z$, nor $\Z^2$. We prove the statement by induction on its Hirsch length $h(G)$.
	Starting off, suppose $h(G) = 2$. This means $G$ is virtually $\Z^2$ (Proposition~\ref{prop:hirsch}), but not $\Z^2$ by our hypothesis. If $G$ is torsion-free, then it is abelian, as all torsion-free virtually abelian nilpotent groups are abelian (see \cite[Lemma 3.1]{koberda2010some}). Because the only torsion-free virtually $\Z^2$ abelian group is $\Z^2$, $G$ must not be torsion-free. Thus, $G$ has torsion and contains a torsion-free subgroup $H$ isomorphic to $\Z^2$. By Proposition \ref{prop:PR_torsion-free}, $G$ is not periodically rigid.
	
	Next, let $G$ be a nilpotent group with $h(G) > 2$. Being nilpotent, $G$ contains a torsion-free finite index nilpotent subgroup, so once again by Proposition \ref{prop:PR_torsion-free} we can suppose that $G$ is torsion-free. In addition, $G$ contains a normal subgroup isomorphic to $\Z$ in its center, which we call $H$. So, $h(G/H) = h(G) - h(H) \geq 2$, and by induction, $G/H$ is not periodically rigid. $G/H$ is also not virtually $\Z$. Finally, by Lemma \ref{lem:PR_quotient}, $G$ is not periodically rigid.
\end{proof}
\begin{corollary}[Theorem~\ref{intro:vnilp}]
	Finitely generated virtually nilpotent groups are periodically rigid if and only if they are not virtually $\Z$ or torsion-free virtually $\Z^2$.
\end{corollary}

\begin{proof}
	Let $G$ be a periodically rigid virtually nilpotent group, and $H$ a finite index torsion-free nilpotent group. If $G$ is torsion-free, by Proposition~\ref{prop:virt_PR}, $H$ has to be periodically rigid. Then, by Theorem~\ref{thm:PR_nil} $H$ is virtually $\Z$ or $\Z^2$, which means $G$ is virtually $\Z$ or torsion-free virtually free $\Z^2$. Suppose $G$ is not torsion-free and not virtually $\Z$. Then, by~\cite[Lemma 13]{bridson2012groups}, there exists an epimorphism $f:H\to\Z^2$. By Proposition~\ref{prop:WA} $H$ admits a weakly aperiodic SFT, and thus by Lemma~\ref{lem:PR_quotient}, $G$ is not periodically rigid. This contradicts our assumption that $G$ was periodically rigid. Therefore $G$ must be virtually $\Z$.
\end{proof}

\begin{theorem}[Theorem~\ref{intro:poly}]
	\label{thm:PR_poly}
	Finitely generated polycyclic groups are periodically rigid if and only if they are not virtually $\Z$ or torsion-free virtually $\Z^2$.
\end{theorem}

\begin{proof}
	Let $G$ be a polycyclic group that is neither virtually $\Z$ nor torsion-free virtually $\Z^2$.
	We proceed one again by induction on the Hirsch length of $G$, $h(G)$. If $h(G) = 2$, then $G$ is virtually $\Z^2$ and has torsion elements. Thus, by Proposition~\ref{prop:PR_torsion-free}, $G$ is not periodically rigid. 
	
	Now, let $h(G) = n > 2$. As $G$ is polycyclic, it contains a torsion-free polycylcic subgroup of finite index. Therefore, by Proposition~\ref{prop:PR_torsion-free}, we can assume $G$ is torsion-free. In addition, as $G$ is polycyclic, it contains a normal subgroup isomorphic to $N = \Z^k$ for $k>0$. If $k\geq2$, take a normal subgroup $H$ isomorphic to $\Z^2$. Then, $G$ satisfies the exact sequence 
	$$1 \to \Z^2\to G\to G/H \to 1,$$
	where $h(G/H) = n - 2>0$. By Lemma~\ref{lem:PR_exact}, $G$ is not periodically rigid because $G/H$ contains a torsion-free element. Finally, if $k = 1$, then $h(G/N) = h(G) - h(N) = n-1 \geq 2$ and $G$ is not periodically rigid by the induction hypothesis and Lemma \ref{lem:PR_quotient}.
\end{proof}


\section*{Acknowledgments}

I would like to thank Nathalie Aubrun for her constant support, comments and discussions. I would also like to thank Etienne Moutot and Solène Esnay for many fruitful discussions on periodically rigid groups.

%
%
%
\printbibliography

\vspace{0.5cm}

\begin{tabular}{@{}l}\scshape Universit\'e Paris-Saclay, CNRS, LISN, Gif-sur-Yvette, France\\\textit{E-mail address: }\href{mailto:nicolas.bitar@lisn.fr}{nicolas.bitar@lisn.fr}\end{tabular}

\end{document}

%% file: preamble.tex
\usepackage[T1]{fontenc}
\usepackage[utf8]{inputenc}
\usepackage{lmodern}
\usepackage[english]{babel}
\usepackage{csquotes}
\usepackage{geometry} 
\usepackage{amsmath,amssymb,amsfonts,amsthm,mathtools}
\usepackage{mathrsfs}

\usepackage{standalone}

\usepackage{graphicx}
\usepackage{caption}
\usepackage{subcaption} 

\makeatletter

\newskip\@bigflushglue \@bigflushglue = -100pt plus 1fil

\def\bigcentering{\let\\\@centercr\rightskip\@bigflushglue%
\leftskip\@bigflushglue
\parindent\z@\parfillskip\z@skip}

\makeatother

\usepackage{booktabs}
\usepackage{array}
\usepackage{paralist} 

\usepackage{fancyhdr}
\usepackage{emptypage} 
 
\fancypagestyle{plain}{
\fancyhf{}
\fancyfoot[RO,RE]{\thepage}

}
\usepackage{ stmaryrd }

\usepackage[Lenny]{fncychap}

\usepackage[nottoc]{tocbibind} 
\usepackage{tocloft}

\usepackage{textcomp} 
\usepackage{xcolor} 
\usepackage{epigraph} 

\usepackage[
colorlinks=true,
urlcolor=purple,
linkcolor=purple!87!black,
citecolor=green!60!black,
pdfborder={0 0 0},
]{hyperref}
\usepackage{bookmark}

\theoremstyle{plain}
\newtheorem{theorem}{Theorem}[section] 
\newtheorem{proposition}[theorem]{Proposition}
\newtheorem{lemma}[theorem]{Lemma}
\newtheorem{corollary}[theorem]{Corollary}
\theoremstyle{definition}
\newtheorem{remark}[theorem]{Remark}
\newtheorem{definition}[theorem]{Definition} 
\newtheorem{example}[theorem]{Example} 
\newtheorem{question}[theorem]{Question}
\newtheorem{conjecture}[theorem]{Conjecture}

\theoremstyle{plain}
\newtheorem{thmx}{Theorem}

\newtheorem{conjx}[thmx]{Conjecture}

\newcommand{\WP}{\textnormal{WP}}

\newcommand{\supp}{\textnormal{supp}}
\newcommand{\co}{\textnormal{co}}

\newcommand{\N}{\mathbb{N}}

\newcommand{\Z}{\mathbb{Z}}

\newcommand{\G}{\mathcal{G}}

\newcommand{\X}{\mathcal{X}}

\newcommand{\F}{\mathbb{F}}
\newcommand{\Fo}{\mathcal{F}}

\newcommand{\llangle}{\langle\langle}
\newcommand{\rrangle}{\rangle\rangle}
\newcommand{\factor}{\sqsubseteq}

\newcommand{\define}[1]{\textbf{#1}}

\newcommand{\Fix}{\textnormal{Fix}}
\newcommand{\upa}[1]{#1^{\uparrow}}
\newcommand{\act}{\curvearrowright}

\newcommand{\Cl}{\textnormal{Cl}}
\newcommand{\Free}{\textnormal{Free}}
\newcommand{\Sub}{\textnormal{Sub}}
\newcommand{\stab}{\textnormal{stab}}
\newcommand{\orb}{\textnormal{orb}}

\providecommand{\keywords}[1]{{\small\textit{Keywords:} #1}}

%% file: figures/grafo1.tex
\begin{tikzpicture}

\draw [black,-latex](0,0) -- (1.9,0);
\draw [black,-latex](2,0) -- (2,1.9);
\draw [black,-latex](2,2) -- (0.1,2);
\draw [black,-latex](0,2) -- (0,0.1);

\draw [black,-latex](0,0) -- (0.89,-0.45);
\draw [black,-latex](1,-0.5) -- (1.93,-0.05);
\draw [black,-latex](2,0) -- (2.455,0.905);
\draw [black,-latex](2.5,1) -- (2.062,1.915);
\draw [black,-latex](2,2) -- (1.1,2.465);
\draw [black,-latex](1,2.5) -- (0.09,2.04);
\draw [black,-latex](0,2) -- (-0.46,1.08);
\draw [black,-latex](-0.5,1) -- (-0.05,0.08);

\draw[fill=white]   (0,0) circle [radius=0.1];
\draw[fill=white]   (2,0) circle [radius=0.1];
\draw[fill=white]   (2,2) circle [radius=0.1];
\draw[fill=white]   (0,2) circle [radius=0.1];

\draw[fill=white]   (1,-0.5) circle [radius=0.1];
\draw[fill=white]   (2.5,1) circle [radius=0.1];
\draw[fill=white]   (1,2.5) circle [radius=0.1];
\draw[fill=white]   (-0.5,1) circle [radius=0.1];

\end{tikzpicture}